\documentclass[reqno]{amsart}
\usepackage{amssymb}
\usepackage{amsfonts}
\usepackage{amsthm}
\usepackage{amsmath}
\usepackage{hyperref}
\newcommand{\N}{{\mathbb{N}}}
\newcommand{\Z}{{\mathbb{Z}}}

\newcommand{\R}{{\mathbb{R}}}

\newcommand{\near}{{\text{\rm{}near}}}
\newcommand{\far}{{\text{\rm{}far}}}
\newcommand{\dod}{{\text{\rm{}DoD}}}
\newcommand{\ntx}{{\nabla_{\!t,x}}}

\let\Re=\undefined\DeclareMathOperator*{\Re}{Re}

\newcommand{\eps}{{\varepsilon}}

\DeclareMathOperator*{\supp}{supp}  \DeclareMathOperator*{\dist}{dist}

\newcommand{\sqrtDelta}{|\nabla|}  % Alternative to \sqrt{-\Delta}

\newcommand{\qtq}[1]{\quad\text{#1}\quad}

\newcommand{\attn}[1]{\marginpar{\footnotesize\raggedright\cite{ATTN}\kern -2.2ex\vrule width 2ex depth 0.5ex height 2ex\relax\,#1}}

\newtheorem{theorem}{Theorem}[section]

\newtheorem{lemma}[theorem]{Lemma}

\newtheorem{corollary}[theorem]{Corollary}

\newtheorem{proposition}[theorem]{Proposition}
\theoremstyle{definition}
\newtheorem{definition}[theorem]{Definition}

\theoremstyle{remark}

\newtheorem*{remarks}{Remarks}

\newcounter{smalllist}

\newenvironment{CI}{\begin{list}{{\ $\bullet$\ }}{%
\setlength{\topsep}{0mm}\setlength{\parsep}{0mm}\setlength{\itemsep}{0mm}%
\setlength{\labelwidth}{0mm}\setlength{\itemindent}{0mm}\setlength{\leftmargin}{0mm}%
\setlength{\labelsep}{0mm} }}{\end{list}}

\numberwithin{equation}{section} \allowdisplaybreaks
\begin{document}

\title[The defocusing energy-supercritical NLW in three space dimensions]{The defocusing energy-supercritical nonlinear wave equation
in three space dimensions}
\author{Rowan Killip}
\address{University of California, Los Angeles}
\author{Monica Visan}
\address{University of California, Los Angeles}

\begin{abstract}
We consider the defocusing nonlinear wave equation $u_{tt}-\Delta u + |u|^p u=0$ in the energy-supercritical regime
$p>4$.  For even values of the power $p$, we show that blowup (or failure to scatter) must be accompanied by blowup of
the critical Sobolev norm.  An equivalent formulation is that solutions with bounded critical Sobolev norm are global
and scatter.  The impetus to consider this problem comes from recent work of Kenig and Merle who treated the case of
spherically-symmetric solutions.
\end{abstract}

\date{\today}

\maketitle

\tableofcontents

%%%%%%%%%%%%%%%%%%%%%%%%%%%%%%%%%%%%%%%%%%%%%%%%%%%%%%%%%%%%%%%%%%%%%%%%%%%%%%%%%%%%%%%%%%%
%
%
%                                   Section
%
%
%%%%%%%%%%%%%%%%%%%%%%%%%%%%%%%%%%%%%%%%%%%%%%%%%%%%%%%%%%%%%%%%%%%%%%%%%%%%%%%%%%%%%%%%%%%

\section{Introduction}

We consider the initial value problem for the defocusing nonlinear wave equation in three space dimensions:
\begin{equation}\label{nlw}
\begin{cases}
\ u_{tt} - \Delta u + F(u) = 0\\
\ u(0)=u_0,  \ u_t(0)=u_1,
\end{cases}
\end{equation}
where the nonlinearity $F(u)=|u|^p u$ is energy-supercritical, that is, $p>4$.  For the sake of simplicity, we restrict attention
to even values of the power $p$ only.

The class of solutions to \eqref{nlw} is left invariant by the scaling
\begin{equation}\label{scaling}
u(t,x)\mapsto \lambda^\frac2{p} u(\lambda t,\lambda x).
\end{equation}
This defines a notion of \emph{criticality}.  More precisely, a quick computation shows that the only homogeneous
$L_x^2$-based Sobolev norm left invariant by the scaling is $\dot H^{s_c}_x\times\dot H^{s_c-1}_x$, where the
\emph{critical regularity} is $s_c:=\frac 32-\frac 2p$. If the regularity of the initial data to \eqref{nlw} is
higher/lower than the critical regularity $s_c$, we call the problem \emph{subcritical/supercritical}.

We consider \eqref{nlw} for initial data belonging to the critical homogeneous Sobolev space, that is, $(u_0, u_1)\in
\dot H^{s_c}_x\times\dot H^{s_c-1}_x$ in the energy-supercritical regime $s_c>1$. We prove that any maximal-lifespan
solution $u$ with the property that $(u,u_t)$ is uniformly bounded (throughout its lifespan) in $\dot
H^{s_c}_x\times\dot H^{s_c-1}_x$ must be global and scatter.

Let us start by making the notion of a solution more precise.

\begin{definition}[Solution]\label{D:solution}
A function $u: I \times \R^3 \to \R$ on a non-empty time interval $0\in I \subset \R$ is a \emph{(strong) solution} to
\eqref{nlw} if $(u,u_t)\in C^0_t (K ;\dot H^{s_c}_x\times\dot H^{s_c-1}_x)$ and $u\in L_{t,x}^{2p}(K \times \R^3)$ for
all compact $K \subset I$, and obeys the Duhamel formula
\begin{equation}\label{old duhamel}
\begin{aligned}
\begin{bmatrix} u(t)\\[1ex] u_t(t)\end{bmatrix}
&= \begin{bmatrix} \cos(t\sqrtDelta) &  \sqrtDelta^{-1} \sin(t\sqrtDelta) \\[1ex]
            -\sqrtDelta\sin(t\sqrtDelta) & \cos(t\sqrtDelta) \end{bmatrix}
                \begin{bmatrix} u(0) \\[1ex] u_t(0)\end{bmatrix} \\
&\qquad - \int_{0}^{t} \begin{bmatrix} |\nabla|^{-1}\sin\bigl((t-s)\sqrtDelta\bigr) \\[1ex]
        \cos\bigl((t-s)\sqrtDelta\bigr) \end{bmatrix} F(u(s))\,ds
\end{aligned}
\end{equation}
for all $t \in I$.  We refer to the interval $I$ as the \emph{lifespan} of $u$. We say that $u$ is a \emph{maximal-lifespan
solution} if the solution cannot be extended to any strictly larger interval. We say that $u$ is a \emph{global solution} if $I=\R$.
\end{definition}

We define the \emph{scattering size} of a solution to \eqref{nlw} on a time interval $I$ by
\begin{equation}\label{E:S defn}
S_I(u):= \int_I \int_{\R^d} |u(t,x)|^{2p}\, dx \,dt.
\end{equation}

Associated to the notion of solution is a corresponding notion of blowup.  By the standard local theory (see
Theorem~\ref{T:local}), this precisely corresponds to the impossibility of continuing the solution.

\begin{definition}[Blowup]\label{D:blowup}
We say that a solution $u$ to \eqref{nlw} \emph{blows up forward in time} if there exists a time $t_1 \in I$ such that
$$ S_{[t_1, \sup I)}(u) = \infty$$
and that $u$ \emph{blows up backward in time} if there exists a time $t_1 \in I$ such that
$$ S_{(\inf I, t_1]}(u) = \infty.$$
\end{definition}

Our main result is the following

\begin{theorem}[Spacetime bounds]\label{T:main}
Suppose $p>4$ is even and let $u:I\times\R^3\to\R$ be a solution to \eqref{nlw} such that
$(u,u_t) \in L_t^\infty (I; \dot H^{s_c}_x \times \dot H^{s_c-1}_x)$.  Then
$$
S_I(u)\leq C\bigl(\|(u, u_t)\|_{L_t^\infty (I; \dot H^{s_c}_x \times \dot H^{s_c-1}_x)}\bigr).
$$
\end{theorem}

We have not considered other values of $p>4$, for which the nonlinearity is no longer a polynomial in $u$; we felt that it would
muddy the main thrust of the argument, without due reward.

As mentioned above, finite-time blowup of a solution to \eqref{nlw} must be accompanied by divergence of the scattering
size defined in \eqref{E:S defn}.  Thus, Theorem~\ref{T:main} immediately implies

\begin{corollary}[Spacetime bounds]\label{C:main} If $u:I\times\R^3\to\R$ is a maximal-lifespan solution to \eqref{nlw} with $(u,u_t)
\in L_t^\infty (I; \dot H^{s_c}_x \times \dot H^{s_c-1}_x)$, then $u$ is global and moreover,
$$
S_\R(u)\leq C\bigl(\|(u, u_t)\|_{L_t^\infty (\R; \dot H^{s_c}_x \times \dot H^{s_c-1}_x)}\bigr).
$$
\end{corollary}

This corollary takes on a more appealing form if we rephrase it in the contrapositive:

\begin{corollary}[Nature of blowup]\label{C:blowup} A solution $u:I\times\R^3\to\R$ to \eqref{nlw} can only blow up in finite time
or be global but fail to scatter if its $\dot H^{s_c}_x \times \dot H^{s_c-1}_x$ norm diverges.
\end{corollary}

For spherically-symmetric initial data, Theorem~\ref{T:main} was proved by Kenig and Merle \cite{kenig-merle:wave sup}.
The (non-radial) analogue of Theorem~\ref{T:main} for NLS in dimensions $d\geq 5$ was proved in \cite{Miel} by adapting the methods of \cite{Berbec}.
We will discuss these papers and their relation to the results presented here more fully when we outline the proof of
Theorem~\ref{T:main}.  Before doing this, let us briefly review some of the backstory and in particular, the origins of
some of the techniques we will be using.

When $p=4$, or equivalently, $s_c=1$, the critical Sobolev norm is automatically bounded in time by virtue of the
conservation of energy:
\begin{equation}\label{energy}
E(u) = \int_{\R^3} \tfrac12|u_t|^2 + \tfrac12|\nabla u|^2 + \tfrac{1}{p+2}|u|^{p+2} \, dx.
\end{equation}
This \emph{energy-critical} case of \eqref{nlw} has received particular attention because of this property.  Global
well-posedness was proved in a series of works \cite{Grillakis:3Dwave,Grillakis:lowDwave,Kapitanskii,Rauch,Struwe:rad,ShatahStruwe:reg,ShatahStruwe}
with finiteness of the scattering size being added later; see \cite{bahouri-gerard,GinibreSofferVelo,Nakanishi,Pecher,tao:Ecrit wave}.  Certain monotonicity formulae,
the Morawetz and energy flux identities, play an important role in all these
results.  It is important that these monotonicity formulae also have critical scaling.

In the energy-supercritical case discussed in this paper, all conservation laws and monotonicity formulae have scaling
below the critical regularity. At the present moment, there is no technology for treating large-data dispersive
equations without some \emph{a priori} control of a critical norm. Indeed, one may assert that the fundamental
difficulty associated with the 3D Navier--Stokes system is controlling the possible growth of (scaling-)critical norms.
This is the purpose of the $L^\infty_t(I; \dot H^{s_c}_x \times \dot H^{s_c-1}_x)$ assumption in Theorem~\ref{T:main};
it plays the role of the missing conservation law at the critical regularity.  Nevertheless, the fact that the monotonicity formulae have
non-critical scaling remains a problem.

The problem of having monotonicity formulae at a different regularity to the critical (coercive) conservation laws is a
difficulty intrinsic to the nonlinear Schr\"odinger equation and it was in this setting that the first methods were
developed for its treatment.  The original breakthrough in this direction was Bourgain's paper \cite{borg:scatter}.
His work introduced the induction on energy technique, which was then further developed in \cite{CKSTT:gwp,tao:gwp radial}. In this
paper, we will use a variant of this method that was introduced by Kenig and Merle, \cite{KenigMerle:H1}, building on
work of Keraani, \cite{keraani-l2}.  In this latter approach, one first shows that failure of the theorem implies the
existence of minimal counterexamples.  This part of the argument is based on concentration-compactness techniques and is very
robust, with very little that is equation-specific.  It breaks the scaling symmetry because such minimal
counterexamples have an intrinsic length scale, albeit time-dependent.  The second part of this approach is to use
conservation laws and/or monotonicity formulae to show that such counterexamples do not exist.  Like the conservation
laws and monotonicity formulae themselves, this part of the argument is intrinsically equation dependent.

\subsection{Outline of the proof}\label{SS:outline}

We argue by contradiction.  Failure of Theorem~\ref{T:main} implies the existence of very special types of
counterexamples.  Such counterexamples are then shown to have a wealth of properties not immediately apparent from
their construction, so many properties, in fact, that they cannot exist.

While we will make some further reductions later, the main property of the special counterexamples is almost periodicity modulo
symmetries:

\begin{definition}[Almost periodicity modulo symmetries]\label{D:ap}
A solution $u$ to \eqref{nlw} with lifespan $I$ is said to be \emph{almost periodic modulo symmetries} if $(u,u_t)$ is bounded
in $\dot H^{s_c}_x \times \dot H^{s_c-1}_x$ and there
exist functions $N: I \to \R^+$, $x:I\to \R^3$, and $C: \R^+ \to \R^+$ such that for all $t \in I$ and $\eta > 0$,
$$
\int_{|x-x(t)| \geq C(\eta)/N(t)} \bigl||\nabla|^{s_c} u(t,x)\bigr|^2\, dx
+ \int_{|x-x(t)| \geq C(\eta)/N(t)} \bigl||\nabla|^{s_c-1} u_t(t,x)\bigr|^2\, dx\leq \eta
$$
and
$$
\int_{|\xi| \geq C(\eta) N(t)} |\xi|^{2s_c}\, | \hat u(t,\xi)|^2\, d\xi
+ \int_{|\xi| \geq C(\eta) N(t)} |\xi|^{2(s_c-1)}\, | \hat u_t(t,\xi)|^2\, d\xi\leq \eta .
$$
We refer to the function $N(t)$ as the \emph{frequency scale function} for the solution $u$, to $x(t)$ as the
\emph{spatial center function}, and to $C(\eta)$ as the \emph{compactness modulus function}.
\end{definition}

\begin{remarks}
1. Given a time $t_0\in I$ we may rescale the function $u(t_0,x)$ so as to renormalize the frequency scale to equal
one.  We may then perform a spatial translation to bring the spatial center of the function to the origin.  Noting that
these operations are symmetries of our equation and incorporating an additional time translation, this procedure yields
a solution to \eqref{nlw} called the \emph{normalization} of $u$ associated to the time $t_0$:
\begin{align}\label{E:Normalization}
u^{[t_0]}(t,x) &:= N(t_0)^{-\frac2p} u\bigl(t_0+t N(t_0)^{-1},x(t_0) + x N(t_0)^{-1}\bigr).
\end{align}
Note that the normalization of $u$ is still almost periodic modulo symmetries; indeed, it admits the same compactness modulus function
as~$u$.

2. By the Ascoli--Arzela Theorem, a family of functions is precompact in $\dot H^{s}_x(\R^3)$ if and only if it is
norm-bounded and there exists a compactness modulus function $C$ so that
$$
\int_{|x| \geq C(\eta)} \bigl||\nabla|^{s} f(x)\bigr|^2\ dx + \int_{|\xi| \geq C(\eta)} |\xi|^{2s} \, |\hat f(\xi)|^2\ d\xi \leq \eta
$$
for all functions $f$ in the family.  Thus, an equivalent formulation of Definition~\ref{D:ap} is as follows: $u$ is almost
periodic modulo symmetries if and only if
\begin{align}\label{E:Standard orbit}
\bigl\{ \bigl( u^{[t_0]}(0), \partial_t u^{[t_0]}(0) \bigr)  : t_0\in I \bigr\}
\end{align}
is a precompact  subset of $\dot H^{s_c}_x\times \dot H^{s_c-1}_x$.

3.  The continuous image of a compact set is compact.  Thus, by Sobolev embedding, almost periodic (modulo symmetries)
solutions obey the following: For each $\eta>0$ there exists $C(\eta)>0$ so that
\begin{equation}\label{E:u' compact}
\bigl\| u(t,x) \bigr\|_{L^\infty_t L_x^{\frac{3p}{2}}(\{|x-x(t)|\geq C(\eta)/N(t)\})}
    + \bigl\| \ntx u(t,x) \bigr\|_{L^\infty_t L_x^{\frac{3p}{p+2}}(\{|x-x(t)|\geq C(\eta)/N(t)\})}
    \leq \eta
\end{equation}
where $\ntx u =(u_t,\nabla u)$ denotes the space-time gradient of $u$.

\end{remarks}

With these preliminaries out of the way, we can now describe the first major milestone in the proof of Theorem~\ref{T:main}.

\begin{theorem}[Reduction to almost periodic solutions, \cite{kenig-merle:wave sup}]\label{T:reduct}
Assume Theorem~\ref{T:main} failed.  Then there exists a maximal-lifespan solution $u:I\times\R^3\to \R$ to \eqref{nlw} such that
$(u,u_t) \in L_t^\infty (I; \dot H^{s_c}_x \times \dot H^{s_c-1}_x)$, $u$ is almost periodic modulo symmetries,
and $u$ blows up both forward and backward in time. Moreover, $u$ is minimal among all blowup solutions in the sense that
$$
\sup_{t\in I} \bigl\| \bigl(u(t), u_t(t)\bigr)\bigr\|_{\dot H^{s_c}_x\times \dot H^{s_c-1}_x}
 \leq \sup_{t\in J}\bigl\| \bigl(v(t), v_t(t)\bigr)\bigr\|_{\dot H^{s_c}_x\times \dot H^{s_c-1}_x}
$$
for all maximal-lifespan solutions $v:J\times\R^3 \to \R$ that blow up in at least one time direction.
\end{theorem}

The reduction to almost periodic solutions is now a standard technique in the analysis of dispersive equations at critical
regularity.  Their existence was first proved by Keraani \cite{keraani-l2} in the context of the mass-critical NLS and first used as
a tool for proving global well-posedness by Kenig and Merle \cite{KenigMerle:H1}.  As noted above, Theorem~\ref{T:reduct} was proved
by Kenig and Merle in \cite{kenig-merle:wave sup}; for other instances of the same techniques, see
\cite{kenig-merle:wave,KenigMerle:H1/2,KKSV,Berbec,ClayNotes,Miel,KTV,KVZ,TVZ:cc,TVZ:sloth}.

We will also need the following further refinement of Theorem~\ref{T:reduct}:

\begin{theorem}[Three special scenarios for blowup, \cite{Berbec}]\label{T:enemies}
Suppose that Theorem~\ref{T:main} failed.  Then there exists a maximal-lifespan solution $u:I\times\R^3\to \R$, which
obeys $(u,u_t) \in L_t^\infty (I; \dot H^{s_c}_x \times \dot H^{s_c-1}_x)$, is almost periodic modulo symmetries, and
$S_I(u)=\infty$.  Moreover, we can also ensure that the lifespan $I$ and the frequency scale function $N:I\to\R^+$
match one of the following three scenarios:
\begin{itemize}
\item[I.] (Finite-time blowup) We have that either $\sup I<\infty $ or $|\inf I|<\infty$.
\item[II.] (Soliton-like solution) We have $I = \R$ and
\begin{equation*}
 N(t) = 1 \quad \text{for all} \quad t \in \R.
\end{equation*}
\item[III.] (Low-to-high frequency cascade) We have $I = \R$,
\begin{equation*}
\inf_{t \in \R} N(t) \geq 1, \quad \text{and} \quad \limsup_{t \to +\infty} N(t) = \infty.
\end{equation*}
\end{itemize}
\end{theorem}

The reference given above discusses the energy-critical NLS; however, the result follows from Theorem~\ref{T:reduct} by
the same arguments since they are essentially combinatorial and so apply to any dispersive equation.  As we are
treating a problem whose critical regularity lies above that of the conserved quantity, the energy-critical NLS serves
as a better model than the mass-critical NLS.  This is the reason for using this set of special scenarios rather than
those obtained in \cite{KTV}.

A further manifestation of the minimality of $u$ as a blowup solution is the absence of a scattered wave at the
endpoints of the lifespan $I$; more formally, we have the following Duhamel formulae, which play an important role in
proving needed decay.  This is a robust consequence of almost periodicity modulo symmetries; see, for example,
\cite{ClayNotes}.

\begin{lemma}[No-waste Duhamel formulae]\label{L:duhamel}
Let $u$ be an almost periodic solution to \eqref{nlw} on its maximal-lifespan $I$.  Then, for all $t\in I$,
\begin{equation}\label{Duhamel}
\begin{aligned}
\begin{bmatrix} u(t)\\ u_t(t)\end{bmatrix}
&= \int_{t}^{\sup I} \begin{bmatrix} \frac{\sin\bigl((t-s)\sqrtDelta\bigr)}{\sqrtDelta}\\ \cos\bigl((t-s)\sqrtDelta\bigr) \end{bmatrix}F(u(s))\,ds\\
&= -\int_{\inf I}^t \begin{bmatrix} \frac{\sin\bigl((t-s)\sqrtDelta\bigr)}{\sqrtDelta}\\ \cos\bigl((t-s)\sqrtDelta\bigr) \end{bmatrix}F(u(s))\,ds
\end{aligned}
\end{equation}
as weak limits in $\dot H^{s_c}_x\times \dot H^{s_c-1}_x$.
\end{lemma}

Representations of this type are central tools for improving the decay and/or regularity properties of the solution
$u$.  For the problem under discussion in this paper, it is better decay that is required since the key monotonicity
formula (the Morawetz identity) and the key conservation law (the energy) have $\dot H^1_x\times L^2_x$ scaling.   We
need access to these identities in order to preclude the soliton-like and frequency-cascade solutions (described in
Theorem~\ref{T:enemies}), respectively. It is considerably easier to dispense with the finite-time blowup solution:
minimality forces the solution to lie inside a light cone, which in turn implies that the energy is zero.  This
argument is presented in Section~\ref{S:ftb} and is little different from the treatment in \cite{kenig-merle:wave sup}.

A key problem in low dimensional situations, such as the problem under discussion in this paper, is making the
integrals \eqref{Duhamel} converge in a better sense, for example, in some norm.  In \cite{kenig-merle:wave sup}, an
incoming/outgoing wave decomposition is used together with the weighted decay available from radial Sobolev embedding.
A not dissimilar technique was used in \cite{KTV,KVZ}, which studied 2D NLS; the 3D NLW has the same poor dispersive
estimate as the 2D NLS.

In this paper, we will prove that the Duhamel integrals converge by making use of the energy flux identity (cf.
Lemma~\ref{L:energy flux}); the same idea (albeit with a different purpose) was used in \cite[Corollary~4.3]{tao:Ecrit wave}.
This is much weaker than the weighed estimates available in the radial case.  Nevertheless, by expanding on the
ideas in our earlier paper \cite{Berbec}, we are able to show that $u$ lies in $L_t^\infty L^q_x(\R^3)$ for some $q<\tfrac{3p}2$,
the exponent given by Sobolev embedding.  This is the topic of Section~\ref{S:AddnlDecay}.  While this does constitute
better decay than the apriori bound, it is not sufficient to use the monotonicity/conservation laws; these require
$L_x^2$-type control on $\ntx u$.

As in \cite{Berbec}, we will employ the double Duhamel identity to upgrade the $L^q_x$ estimates to the better kind of
decay required. This identity was first introduced in \cite{CKSTT:gwp} for the nonlinear Schr\"odinger equation and
results from taking the inner product between the two formulae in \eqref{Duhamel}.  We have not seen this technique
used before for the nonlinear wave equation and, in light of this, it is perhaps worth noting that in order to maintain
the natural structure of the formula, one should take the inner-product of the two representations of the spacetime derivative $\ntx u(0)$,
rather than of $u(0)$ as in the Schr\"odinger case.  This is a manifestation of the fact that
\eqref{nlw} is second-order in time.

The double Duhamel integrals have very poor convergence properties.  In \cite{Berbec}, this restricted us to working in
five or more (spatial) dimensions. Due to the different nature of the dispersive estimate, this would be analogous to
dimensions six or higher for the wave equation.  Making the double Duhamel formula converge is quite an undertaking, as
we will describe.  First, we localize in space.  This was done already in \cite{CKSTT:gwp} and results in an
improvement equivalent to lowering the applicable dimension by two, which is still insufficient for the nonlinear wave
equation in three dimensions.

In the manner we have described it so far, the space-localized double Duhamel formula reads as follows:
\begin{equation}\label{fake monster}
\begin{aligned}
\int_{\R^3} \chi(x) \bigl| \ntx u(0)\bigr|^2 \,dx
    ={} & - \int_0^\infty \int_{-\infty}^0 \bigl\langle \nabla \tfrac{\sin(|\nabla|t)}{|\nabla|} F(t),\ %
        \chi \nabla \tfrac{\sin(|\nabla|\tau)}{|\nabla|} F(\tau)\bigr\rangle\,d\tau\,dt \\
&     - \int_0^\infty \int_{-\infty}^0 \bigl\langle \cos(|\nabla|t) F(t),\ %
        \chi \cos(|\nabla|\tau) F(\tau)\bigr\rangle\,d\tau\,dt,
\end{aligned}
\end{equation}
where $F(t)$ is short-hand for $F(u(t))$, the nonlinearity in \eqref{nlw}, $\chi$ denotes a spatial cutoff function,
and the inner products are in vector- and scalar-valued $L^2_x(\R^3)$, respectively.  When actually used in
Section~\ref{S:finite E}, there will be additional frequency projections and, in the first occurrence, (fractional)
differential operators.

Our technique for making the integrals in \eqref{fake monster} converge is inspired by a consideration of geometric
optics: In a sense, \eqref{fake monster} represents the nonlinearity $F(t)$ at time $t$ looking at the nonlinearity
$F(\tau)$ at time $\tau$ though a `keyhole' whose aperture is the support of $\chi$. Note that the main part of the
nonlinearity $F(t)$ lies near the center of the wave-packet at that time, namely, $x(t)$.  This indicates the path we
will follow: (a) make sure that points near $x(t)$ cannot see points near $x(\tau)$, at least not directly, (b) control
the amount of diffraction associated with the aperture, and (c) control the contribution from points far from $x(t)$
and/or $x(\tau)$.

For short times, part (a) of this programme is immediate from finite speed of propagation.  If the aperture is far from
$x(0)$, then $x(t)$ and $x(\tau)$ do not have time to travel far enough to see one another.  While this very naive
picture continues to hold in the long-time regime, that is, $x(t)$ cannot catch up to a light ray emanating from
$x(\tau)$ that passes through the aperture, there is no improvement with the passing of time that might allow the time
integral to converge.  We thus need to show that $x(t)$ and $x(\tau)$ travel strictly slower than light.  This is the
topic of Section~\ref{S:subluminal}.  In previous work on critical dispersive equations, the motion of $x(t)$ has been
constrained by using conservation laws, specifically, the conservation of momentum.  This approach is not available in
this case, since we need to control $x(t)$ first in order to obtain finiteness of the conserved quantities.  In the
case of radial data, $x(t)\equiv 0$.

Part (b) of the programme outlined above is encapsulated in Proposition~\ref{P:weakD}.  Thanks to the sub-luminality
proved in Section~\ref{S:subluminal}, we need not consider very long times.

The proof that $u$ (and so also $F(u)$) decays quickly away from $x(t)$ is the subject of Section~\ref{S:quant}, which
handles part (c) of our programme. To the best of our knowledge, this is the first instance when power-law decay has
been obtained without the benefit of radial initial data in the setting of critical dispersive equations.  Moreover, we
obtain this decay in a scaling-invariant space.  The argument uses the energy flux identity to make the Duhamel
integral converge. This places the long-time piece in $L_x^q$ for $q>3p/2$.  To compensate for this, we interpolate
with the decay estimates obtained in Section~\ref{S:AddnlDecay}, which show that $u\in L^\infty_t L^q_x$ for some $q<3p/2$.

In Section~\ref{S:finite E}, we pull these three threads together to prove that not only does $u$ have finite energy,
but even that the energy decays with a power-law away from $x(t)$.  This is done using an iterative procedure that
takes one from $s_c$ derivatives to $s_c-\eps$ derivatives to $s_c-2\eps$ derivatives and so forth.  The last step, from $1+\eps$
derivatives to finite energy, is treated separately, because this can be done much more simply --- $\nabla$ is local in
space, while $|\nabla|^{1+\eps}$ is~not.

Sections~\ref{S:no cascade} and~\ref{S:no soliton} use the finiteness of the energy to show, respectively, that
frequency-cascade and soliton-like solutions to \eqref{nlw} are not possible.
We show that the frequency-cascade solution is inconsistent with the conservation of energy; of course, this would be
meaningless had we not first proved that the energy is finite.  The existence of solitons is precluded by use of the
Morawetz identity (cf. \cite{Morawetz75,MorawetzStrauss}):
\begin{equation*}
\begin{aligned}
&\frac{d\ }{dt} \int_{\R^3} -a_j(x)u_t(t,x) u_j(t,x) - \tfrac12 a_{jj}(x) u(t,x) u_t(t,x) \,dx \\
    ={}& \int_{\R^3} a_{jk}(x) u_j(t,x)u_k(t,x) + \tfrac{p}{2p+4} a_{jj}(x) u(t,x)^{p+2} - \tfrac12a_{jjkk}(x) u(t,x)^2 \,dx
\end{aligned}
\end{equation*}
where $u$ is a solution to \eqref{nlw}, subscripts indicate partial derivatives, and repeated indices are summed.  More
precisely, we use the special case $a(x)=|x|$, which together with the Fundamental Theorem of Calculus and Hardy's
inequality yields
\begin{equation}\label{Morawetz}
\int_I \int_{\R^3} \frac{|u(t,x)|^{p+2}}{|x|} \,dx\,dt \lesssim \|\ntx u\|_{L^\infty_tL^2_x(I\times\R^3)} ^2.
\end{equation}
Notice that by finite speed of propagation, the left-hand side should grow logarithmically in time.

The finite-time blowup solution is precluded in Section~\ref{S:ftb} and does rely on Sections~\ref{S:subluminal} through~\ref{S:finite E}.
Like the frequency-cascade, this type of solution is inconsistent with the conservation of energy; finiteness of the energy in this case
follows from the fact that finite-time blowup solutions are compactly supported at each time.  The idea of using a second (non-critical) conservation law to control
the growth/decay of $N(t)$ originates in the study of NLS (cf. \cite[\S4]{borg:scatter}); in this paper, the assumed boundedness of
the critical Sobolev norm acts as a first conservation law.

\subsection*{Acknowledgements}
The first author was supported by NSF grant DMS-0701085.  The second author was supported by NSF grant DMS-0901166.

%%%%%%%%%%%%%%%%%%%%%%%%%%%%%%%%%%%%%%%%%%%%%%%%%%%%%%%%%%%%%%%%%%%%%%%%%%%%%%%%%%%%%%%%%%%
%
%
%                                   Section
%
%
%%%%%%%%%%%%%%%%%%%%%%%%%%%%%%%%%%%%%%%%%%%%%%%%%%%%%%%%%%%%%%%%%%%%%%%%%%%%%%%%%%%%%%%%%%%

\section{Notation and useful lemmas}

We write $X\lesssim Y$ to indicate that $X\leq C Y$ for some constant $C$, which may change from line to line.
Dependencies will be indicated with subscripts, for example, $X\lesssim_u Y$.  We will write $X\sim Y$ to indicate that $X\lesssim Y\lesssim X$.

Let $\varphi(\xi)$ be a radial bump function supported in the ball $\{ \xi \in \R^d: |\xi| \leq \tfrac {11}{10} \}$ and equal to
$1$ on the ball $\{ \xi \in \R^d: |\xi| \leq 1 \}$.  For each number $N > 0$, we define the Fourier multipliers
\begin{align*}
\widehat{P_{\leq N} f}(\xi) &:= \varphi(\xi/N) \hat f(\xi)\\
\widehat{P_{> N} f}(\xi) &:= (1 - \varphi(\xi/N)) \hat f(\xi)\\
\widehat{P_N f}(\xi) &:= (\varphi(\xi/N) - \varphi(2\xi/N)) \hat f(\xi)
\end{align*}
and similarly $P_{<N}$ and $P_{\geq N}$.  We also define
$$ P_{M < \cdot \leq N} := P_{\leq N} - P_{\leq M} = \sum_{M < N' \leq N} P_{N'}$$
whenever $M < N$.  We will only have cause to use these multipliers when $M$ and $N$ are \emph{dyadic numbers} (that is, of the form $2^n$
for some integer $n$); in particular, all summations over $N$ or $M$ are understood to be over dyadic numbers.

Like all Fourier multipliers, the Littlewood-Paley operators commute with derivatives and the propagator.  We will only need the
basic properties of these operators, particularly,

\begin{lemma}[Bernstein estimates]\label{Bernstein}
 For $1 \leq p \leq q \leq \infty$,
\begin{align*}
\bigl\| |\nabla|^{\pm s} P_N f\bigr\|_{L^p_x} &\sim N^{\pm s} \| P_N f \|_{L^p_x},\\
\|P_{\leq N} f\|_{L^q_x} &\lesssim N^{\frac{3}{p}-\frac{3}{q}} \|P_{\leq N} f\|_{L^p_x},\\
\|P_N f\|_{L^q_x} &\lesssim N^{\frac{3}{p}-\frac{3}{q}} \| P_N f\|_{L^p_x}.
\end{align*}
\end{lemma}

In three space dimensions, the wave equation obeys the strong form of the Huygens principle.  This is most easily expressed
in terms of the explicit form of the propagator:

\begin{lemma}\label{L:propagator} For Schwartz functions $f$,
\begin{equation}\label{E:propagator}
\Bigl[ \tfrac{\sin(t\sqrtDelta)}{\sqrtDelta} f \Bigr](x) = \tfrac1{4\pi t} \int_{|x-y|=t} f(y) \,dS(y)
\end{equation}
where $dS$ denotes the usual $2$-dimensional surface measure.
\end{lemma}

To be absolutely clear about the normalization here, we note that if $f\equiv 1$, then
$\text{RHS\eqref{E:propagator}}\equiv t$. A well-known consequence of \eqref{E:propagator} is the following:

\begin{lemma}[Dispersive estimate]\label{L:Dispersive estimate}
For $2\leq q \leq \infty$ and $f\in L^{q'}(\R^3)$,
$$
 \Bigl\| \sqrtDelta^{-1}\sin(t\sqrtDelta) f \Bigr\|_{L^q(\R^3)} \lesssim
    |t|^{-(1-\frac{2}{q})} \bigl\| |\nabla|^{-\frac4q} \nabla f \bigr\|_{L^{q'}(\R^3)}
$$
\end{lemma}

\begin{proof}
For $q=2$, the result reduces to the boundedness of Fourier multiplier $\sin(t|\xi|)$.  When $q=\infty$, the desired
estimate is
$$
 \bigl\| \sqrtDelta^{-1}\sin(t\sqrtDelta) f \bigr\|_{L^\infty(\R^3)} \lesssim_q |t|^{-1} \| \nabla f\|_{L^{1}(\R^3)},
$$
which follows easily from \eqref{E:propagator} and
$$
\int_{|\omega|=1} f(t\omega) \,dS(\omega) = \int_t^\infty \int_{|\omega|=1} -\omega\cdot\nabla f(r\omega) \,dS(\omega)\,dr
    \leq t^{-2} \|\nabla f\|_{L^1(\R^3)}.
$$
For general $q$ one may apply the theory of analytic interpolation.  We caution the reader, however, that this requires
the use of BMO and its interpolation theory; see, for example, \cite[\S IV.5]{stein:large}, which uses a very closely related
estimate as the motivating example.
\end{proof}

Note that even when $q=\infty$, this gives only $t^{-1}$ decay. This is not integrable in time and so insufficient to prove
convergence of the Duhamel formulae \eqref{Duhamel}.

As is now well understood,  the dispersive estimate forms the basis for proving Strichartz estimates, which we record next.
See \cite{tao:keel,Pecher, Strichartz77} and the references therein for further information.

\begin{lemma}[Strichartz estimates]\label{L:Strichartz}
Let $I$ be a compact time interval and let $u: I\times\R^3 \to \mathbb \R$ be a solution to the forced wave equation
$$
u_{tt}- \Delta u + F = 0.
$$
Then for any $t_0\in I$ and $6\leq q < \infty$,
\begin{align*}
\|u\|_{L_t^\infty \dot H^{s_c}_x} + \|u_t\|_{L_t^\infty \dot H^{s_c-1}_x} &+ \|u\|_{L_{t,x}^{2p}}
    + \|u\|_{L_t^{p/2} L_x^\infty} + \|u\|_{L_t^\infty L_x^{3p/2}} \\
    &+ \|\nabla u\|_{L_t^\frac{4p}{p-4} L_x^\frac{4p}{p+4}} + \|\nabla u\|_{L_t^\infty L_x^\frac{3p}{p+2}}  \\
& {\mkern -100mu} \lesssim \|u(t_0)\|_{\dot H^{s_c}_x} + \|u_t(t_0)\|_{\dot H^{s_c-1}_x} + \bigl\| |\nabla|^{s_c} F\bigr\|_{L_t^{\frac{2q}{q+6}}L_x^{\frac q{q-1}}},
\end{align*}
where all spacetime norms are on $I\times\R^3$.
\end{lemma}

As noted in the Introduction, much of the argument presented here was inspired by work on NLS.  The most favourable
difference between NLW and NLS is that NLW enjoys finite speed of propagation.  Unfortunately, we will need to deal
with a non-integer number of derivatives, which is inherently a non-local operator.  To cope with this, we will make use
of the following:

\begin{lemma}[Mismatch estimates]\label{L:mismatch}
Let $\phi_1$ and $\phi_2$ be (smooth) functions obeying
$$
|\phi_j| \leq 1 \qtq{and} \dist( \supp \phi_1,\, \supp \phi_2 ) \geq A,
$$
for some large constant $A$.  Then for $\sigma>0$ and $1\leq p\leq q\leq \infty$,
\begin{align}\label{mismatch est}
\bigl\| \phi_1 |\nabla|^\sigma P_{\leq 1} (\phi_2 f) \bigr\|_{L^q_x(\R^3)}
    + \bigl\| \phi_1 \nabla |\nabla|^{\sigma-1} P_{\leq 1} (\phi_2 f) \bigr\|_{L^q_x(\R^3)}
    \lesssim A^{-\sigma-\frac3p + \frac3q} \|\phi_2 f\|_{L^p_x(\R^3)}
\end{align}
\end{lemma}

\begin{proof}
Elementary computations show that if $K(x)$ denotes the convolution kernel associated to either of the Fourier multipliers
$|\nabla|^{\sigma} P_{\leq1}$ or $\nabla |\nabla|^{\sigma-1} P_{\leq1}$, then $|K(x)|\lesssim |x|^{-3-\sigma}$.

Noting that
$$
\int_A^\infty  r^{-(3+\sigma) \tilde q} r^2\,dr \lesssim_{\tilde q} A^{-\tilde q(\sigma + \frac3p - \frac3q)}
    \qtq{when} \tfrac1p+\tfrac1{\tilde q} = 1 + \tfrac1q,
$$
the result follows from Young's inequality.
\end{proof}

As discussed in the introduction, our next proposition shows that waves do not diffract too much through a large aperture.  We are content
to show that there is sufficient decay in \eqref{E:weakD} below and have made no attempt to find the optimal bound.

\begin{proposition}[Weak diffraction]\label{P:weakD}
Let $\phi:\R^3\to [0,1]$ be a smooth compactly supported function such that $\phi(x)=1$ for $|x|<1$ and $\phi(x)=0$ for $|x|>\tfrac{11}{10}$.
Also, let $\theta:\R^3\to [0, \infty)$ be defined by
\begin{align}\label{theta}
\theta(\xi) = \prod_{j=1}^3 \Bigl(\frac{\sin(\xi_j)}{\xi_j}\Bigr)^{4}.
\end{align}
Then
\begin{align}\label{E:weakD}
&\Bigl|\iint \bigl\langle \nabla \tfrac{\sin(|\nabla|t)}{|\nabla|} \theta(i\nabla) F(t),\;\phi\bigl(\tfrac{\cdot}{R}\bigr)
    \nabla \tfrac{\sin(|\nabla|\tau)}{|\nabla|}\theta(i\nabla) G(\tau)\bigr\rangle\,d\tau\,dt\\
&+\iint \bigl\langle \cos(|\nabla|t)\theta(i\nabla) F(t),\;\phi\bigl(\tfrac{\cdot}{R}\bigr)
    \cos(|\nabla|\tau)\theta(i\nabla) G(\tau)\bigr\rangle\,d\tau\,dt\Bigr|
\lesssim R^{-\frac1{10}} \| F \|_{L^\infty_t L^1_x} \| G \|_{L^\infty_\tau L^1_y},\notag
\end{align}
provided $F(t,x)$ and $G(\tau,y)$ are supported where
\begin{align}\label{hyp}
|t|+|\tau|+|x|+|y|\lesssim R \quad \text{and} \quad \frac{|t-\tau|-|x-y|}R\gtrsim 1
\end{align}
for any (large) constant $R$.
\end{proposition}

\begin{proof}
LHS\eqref{E:weakD} can be rewritten as
$$
\iint \! \iint [K_1(t,x;\tau,y) + K_2(t,x;\tau,y)] F(t,x) G(\tau,y)\,dx\,dy\,d\tau\,dt,
$$
where $K_1$ and $K_2$ are the kernels
$$
K_1:=\Re \int_{\R^3}\int_{\R^3} e^{it|\xi|-i\tau|\eta|+ix\cdot\xi- iy\cdot\eta} R^3 \check{\phi}\bigl( R(\xi-\eta)\bigr)\theta(\xi)\theta(\eta)
    \,d\xi\,d\eta
$$
and
$$
K_2:= - \int_{\R^3}\int_{\R^3} \sin(t|\xi|)\sin(\tau|\eta|) R^3 \check{\phi}\bigl( R(\xi-\eta)\bigr) \theta(\xi)\theta(\eta)
    \bigl[ 1 - \tfrac{\xi\cdot\eta}{|\xi| |\eta|}\bigr] e^{ix\cdot\xi- iy\cdot\eta} \,d\xi\,d\eta.
$$

To prove the proposition, it suffices to estimate these kernels in $L^1_{t,\tau} L^\infty_{x,y}$ on the support of $F$ and $G$.  In view of \eqref{hyp},
this in turn can be effected by proving $R^{-21/10}$ bounds on the kernels in $L^\infty_{t,\tau,x,y}$, which is essentially what we will do.  As a first step, we change
variables to $\mu=\frac{\xi+\eta}2$ and $\nu=\frac{\xi-\eta}2$ and we split the integral into several pieces, first by introducing the cutoffs
$\phi(R^{7/9}\mu)$ and $1-\phi(R^{7/9}\mu)$, and then in the latter case, the cutoffs $\phi(R^{8/9}\nu)$ and $1-\phi(R^{8/9}\nu)$.

When the cutoff $\phi(R^{7/9}\mu)$ is present, we bring absolute values inside the integrals and bound the corresponding contribution (to $K_1$ or $K_2$) by
\begin{align}\label{A}
\lesssim \int_{\R^3}\int_{\R^3} \phi(R^{7/9}\mu) R^3 |\check{\phi}(2R\nu)| \,d\mu\, d\nu \lesssim R^{-7/3}.
\end{align}

When the factor $[1-\phi(R^{7/9}\mu)][1-\phi(R^{8/9}\nu)]$ is present, we again bring the absolute values inside.  We also make use of the fact
that $|\check{\phi}(\zeta)|\lesssim_m |\zeta|^{-m}$ for all $m\in \N$.  Thus, the contribution of this term to either kernel is
\begin{align}\label{B}
&\lesssim \int_{\R^3}\int_{\R^3} \bigl[1-\phi(R^{7/9}\mu)\bigr]\bigl[1-\phi(R^{8/9}\nu)\bigr]
    R^3 |\check{\phi}(2R\nu)| \theta(\mu+\nu)\theta(\mu-\nu)\, d\mu\, d\nu \notag\\
&\lesssim \int_{\R^3} \bigl[1-\phi(R^{8/9}\nu)\bigr] R^3 \bigl|\check{\phi}(2R\nu)\bigr|\, d\nu\notag\\
&\lesssim \int_{|\nu|\gtrsim R^{-8/9}} \frac{R^3} {(R|\nu|)^{24}}\, d\nu\notag\\
&\lesssim R^{-7/3}.
\end{align}

It remains to estimate the contributions associated with the factor $[1-\phi(R^{7/9}\mu)]\phi(R^{8/9}\nu)$.  Here we must treat the two
kernels separately.  We begin by considering $K_1$ which amounts to estimating the integral
$$
I_1:= \int_{\R^3}\int_{\R^3} e^{iR\varphi(\mu, \nu)} \psi(\mu,\nu) \, d\mu\, d\nu,
$$
with
$$
\varphi(\mu, \nu):= \tfrac1R\bigl[t |\mu+\nu|-\tau|\mu-\nu|+(x-y)\cdot \mu+ (x+y)\cdot \nu\bigr]
$$
and
$$
\psi(\mu,\nu):= \bigl[1-\phi(R^{7/9}\mu)\bigr]\phi(R^{8/9}\nu) R^3 \check{\phi}(2R\nu)\theta(\mu+\nu)\theta(\mu-\nu).
$$
To do this, we will employ the technique of non-stationary phase.
For all multi-indices $\alpha\in \Z^3_{\geq 0}$ of length $|\alpha|\leq 4$, we have the following symbol-type estimates:
\begin{gather*}
|\partial_\xi^\alpha \theta(\xi)|\lesssim_\alpha |\xi|^{-|\alpha|},
    \qquad |\partial_\mu^\alpha \psi|\lesssim_\alpha R^3|\mu|^{-|\alpha|},
    \qquad |\partial_\mu^\alpha \varphi|\lesssim_\alpha |\mu|^{1-|\alpha|}, \\
\Bigl|\partial_\mu^\alpha \frac{\mu}{|\mu|}\Bigr|\lesssim_\alpha |\mu|^{-|\alpha|},
    \qtq{and}
    \Bigl| \frac{\mu}{|\mu|}\cdot \nabla_{\!\mu} \varphi\Bigr|\gtrsim 1,
\end{gather*}
uniformly for $(\mu,\nu)\in\supp(\psi)$, which implies $|\mu|\gtrsim R^{-7/9}$ and $|\nu|\lesssim R^{-8/9}$.
To derive the last estimate, one squares both sides and uses the fact that for these $\mu$, $\nu$,
$$
\sqrt{1-\tfrac{|\nu|^2}{|\mu+\nu|^2}}\leq \tfrac{\mu\cdot (\mu\pm \nu)}{|\mu||\mu\pm \nu|}\leq 1,
$$
to obtain
$$
\tfrac{\mu}{|\mu|}\cdot \nabla_{\!\mu} \varphi=\tfrac{t-\tau}R + \tfrac{\mu}{|\mu|}\cdot\tfrac{x-y}R + O\bigl( \tfrac{R^{-2/9}(|t|+|\tau|)}R\bigr),
$$
The inequality now follows from this and \eqref{hyp}.

Using these estimates and the quotient rule in the symbol calculus, we find that the vector
$a:=\bigl(\tfrac{\mu}{|\mu|}\cdot \nabla_{\!\mu} \varphi\bigr)^{-1}\tfrac{\mu}{|\mu|}$ obeys
$$
|\partial_\mu^\alpha a|\lesssim_\alpha |\mu|^{-|\alpha|} \quad \text{uniformly in} \quad  |\mu|\gtrsim R^{-7/9} \text{ and } |\nu|\lesssim R^{-8/9}
$$
when $|\alpha|\leq 4$.  Moreover,
$$
\sup_\nu\bigl|(i\nabla_{\!\mu} \cdot a)^4 \psi\bigr|
\lesssim \sum_{|\alpha_1+\cdots +\alpha_4+\beta|=4} |\partial_\mu^{\alpha_1} a|\cdots |\partial_\mu^{\alpha_4} a||\partial_\mu^\beta \psi|
\lesssim R^3|\mu|^{-4},
$$

Thus, as $e^{iR\varphi}= R^{-4} ( a\cdot i\nabla_{\!\mu})^4 e^{iR\varphi}$, we obtain
\begin{align}\label{I1}
|I_1|&\lesssim R^{-\frac83} \sup_{|\nu|\lesssim R^{-8/9}} \Bigl|\int_{\R^3} e^{iR\varphi(\mu, \nu)} \psi(\mu,\nu) \, d\mu \Bigr|\notag\\
&\lesssim R^{-4-\frac83} \sup_{|\nu|\lesssim R^{-8/9}} \Bigl|\int_{\R^3} e^{iR\varphi(\mu, \nu)} (i\nabla_{\!\mu} \cdot a )^{4}
    \psi(\mu,\nu) \, d\mu \Bigr| \notag\\
&\lesssim R^{-4-\frac83} \int_{|\mu|\gtrsim R^{-7/9}} R^3 |\mu|^{-4}\, d\mu\notag\\
&\lesssim R^{-7/3}.
\end{align}

Collecting \eqref{A}, \eqref{B}, and \eqref{I1}, we obtain
$$
|K_1(t,x;\tau,y)| \lesssim R^{-7/3}.
$$
By virtue of \eqref{hyp}, this settles the $K_1$ portion of the proposition.

We turn now to estimating the remaining portion of the integral defining $K_2$, namely,
\begin{align}
I_2:= \int_{\R^3}\int_{\R^3} \sin(t|\mu+\nu|)&\sin(\tau|\mu-\nu|) R^3 \check{\phi}\bigl( 2R\nu\bigr) e^{i\mu(x-y)+i\nu(x+y)}
    \theta(\mu+\nu)\theta(\mu-\nu)  \notag \\
    & \times \bigl[ 1 - \tfrac{(\mu+\nu)\cdot(\mu-\nu)}{|\mu+\nu| |\mu-\nu|}\bigr] [1-\phi(R^{7/9}\mu)]\phi(R^{8/9}\nu) \,d\mu\,d\nu.
\end{align}
To continue, we use the simple identity
$$
\sin(t|\mu+\nu|)\sin(\tau|\mu-\nu|) = \tfrac12 \Re e^{it|\mu+\nu|-i\tau|\mu-\nu|}-\tfrac12 \Re e^{it|\mu+\nu|+i\tau|\mu-\nu|},
$$
which naturally breaks $I_2$ into the sum of two pieces.  The first summand can be estimated in a manner similar to that used to treat $I_1$,
or by a simplified version of the technique we will use to estimate the second summand,
\begin{align*}
I_2' := \int_{\R^3}\int_{\R^3} e^{i\varphi} R^3 \check\phi(2R\nu)\theta(\mu+\nu)\theta(\mu-\nu)
    \bigl[ 1 - \tfrac{(\mu+\nu)\cdot(\mu-\nu)}{|\mu+\nu| |\mu-\nu|}\bigr] [1-\phi(R^{7/9}\mu)]\phi(R^{8/9}\nu) \,d\mu\,d\nu,
\end{align*}
where
\begin{align*}
\varphi := t|\mu+\nu|+\tau|\mu-\nu| + \mu(x-y) + \nu(x+y).
\end{align*}

As a first estimate on $I_2'$, we note that when $|\nu|\ll |\mu|$, which is where the integrand is supported,
\begin{equation}
1 - \tfrac{(\mu+\nu)\cdot(\mu-\nu)}{|\mu+\nu| |\mu-\nu|} = O\bigl( \tfrac{|\nu|^2}{|\mu|^2}\bigr)
\end{equation}
and hence
\begin{align}
|I_2'| & \lesssim \int_{\R^3} \int_{\R^3} R^3 |\check\phi(2R\nu)| \theta(\mu+\nu)\theta(\mu-\nu) \tfrac{|\nu|^2}{|\mu|^2} \,d\mu\,d\nu
    \lesssim \int_{\R^3} R^3 |\check\phi(2R\nu)| |\nu|^2 \,d\nu  \notag \\
&\lesssim R^{-2}. \label{I2' 1st est}
\end{align}
Note that this is not quite good enough: when integrated over $|t|\lesssim R$ and $|\tau|\lesssim R$, there are no powers of $R$
left over to provide the required decay.  Nevertheless, it does provide the requisite $R^{-1/10}$ bound on the $L^1_{t,\tau}L^\infty_{x,y}$
norm of $K_2$ in the restricted region where $|t+\tau|\leq R^{9/10}$.

It remains only to estimate $I_2'$ in the region where $|t+\tau|\geq R^{9/10}$.  To do this, we will break the $\mu$ integral into six pieces, one near each
coordinate semi-axis.  To this end, let us take a smooth partition of unity of the unit sphere adapted to the open cover
\begin{gather*}
\{\max(|\mu_1|,|\mu_2|) < \tfrac98 \mu_3\},\ \{\max(|\mu_1|,|\mu_2|) < -\tfrac98\mu_3\},\ \{\max(|\mu_3|,|\mu_1|) < \tfrac98\mu_2\},\\
\{\max(|\mu_3|,|\mu_1|) < -\tfrac98\mu_2\}, \ \{\max(|\mu_2|,|\mu_3|) < \tfrac98\mu_1\},\ \{\max(|\mu_2|,|\mu_3|) < -\tfrac98\mu_1\}.
\end{gather*}
We break the $\mu$ integral into pieces by introducing cutoffs $\chi(\mu/|\mu|)$ where $\chi$ denotes one of the elements of this partition of unity.
By symmetry, it suffices to treat the piece associated to the first region listed above.

Recalling that we are considering only the case where $|t+\tau|\geq R^{9/10}$, $|\mu|\gtrsim R^{-7/9}$ and $|\nu|\lesssim R^{-8/9}$, we have
$$
\partial_{\mu_1}^2 \phi = (t+\tau) \tfrac{|\mu|^2-\mu_1^2}{|\mu|^{3}} + O\Bigl(|\mu|^{-1} R^{1-\tfrac19}\Bigr).
$$
Noting that $|\mu|-|\mu_1|\sim |\mu|\sim\mu_3$ on the support of $\chi(\mu/|\mu|)$, we deduce that on this set,
\begin{equation}
\bigl|\partial_{\mu_1}^2 \phi\bigr| \gtrsim \frac{|t+\tau|}{\mu_3}.
\end{equation}
Thus, by writing
$$
\psi(\mu,\nu):= \theta(\mu+\nu)\theta(\mu-\nu)\bigl[ 1 - \tfrac{(\mu+\nu)\cdot(\mu-\nu)}{|\mu+\nu| |\mu-\nu|}\bigr] [1-\phi(R^{7/9}\mu)]
\chi\bigl(\tfrac{\mu}{|\mu|}\bigr)
$$
and noting that
\begin{align*}
\|\psi\|_{L^\infty(d\mu_1)} \lesssim \langle \mu_3 \rangle^{-8} \tfrac{|\nu|^2}{\mu_3^2}
    \qtq{and}
\bigl\|\partial_{\mu_1} \psi\bigr\|_{L^1(d\mu_1)} \lesssim \langle \mu_3 \rangle^{-8} \bigl( \tfrac{|\nu|^2}{\mu_3} + \tfrac{|\nu|^2}{\mu_3^2}  \bigr),
\end{align*}
the Van der Corput Lemma (cf. \cite[p. 334]{stein:large}) yields
\begin{align*}
 \iiint e^{i\varphi} \psi \,d\mu_1\,d\mu_2\,d\mu_3
&\lesssim \iint_{|\mu_2|\lesssim \mu_3} \bigl( \tfrac{\mu_3}{|t+\tau|} \bigr)^{1/2}
    \Bigl\{ \|\psi\|_{L^\infty(d\mu_1)} + \bigl\|\partial_{\mu_1} \psi\bigr\|_{L^1(d\mu_1)} \Bigr\} d\mu_2\,d\mu_3 \\
&\lesssim \tfrac{|\nu|^2}{|t+\tau|^{1/2}},
\end{align*}
uniformly for $|\nu|\lesssim R^{-8/9}$.  Thus when $|t+\tau|\geq R^{9/10}$, we may bound the portion of $I_2'$ partitioned off by $\chi(\mu/|\mu|)$
as follows:
\begin{align*}
& \biggl| \int_{\R^3}\int_{\R^3} e^{i\varphi(\mu,\nu)} \psi(\mu,\nu) R^3 \check\phi(2R\nu) \phi(R^{8/9}\nu) \,d\mu\,d\nu \biggr|
    \lesssim R^{-2} |t+\tau|^{-1/2}.
\end{align*}
As a consequence, we can bound the $L^1_{t,\tau}L^\infty_{x,y}$ norm of $K_2$ on this set of times by $R^{-9/20}$.

This completes the proof of the proposition.
\end{proof}

%%%%%%%%%%%%%%%%%%%%%%%%%%%%%%%%%%%%%%%%%%%%%%%%%%%%%%%%%%%%%%%%%%%%%%%%%%%%%%%%%%%%%%%%%%%
%
%
%                                   Section
%
%
%%%%%%%%%%%%%%%%%%%%%%%%%%%%%%%%%%%%%%%%%%%%%%%%%%%%%%%%%%%%%%%%%%%%%%%%%%%%%%%%%%%%%%%%%%%

\section{NLW background}\label{S:background}

We start by recording the standard local well-posedness theory for \eqref{nlw}.  All results follow from the Strichartz inequalities discussed in Lemma~\ref{L:Strichartz}
and the usual contraction mapping arguments.

\begin{theorem}[Local well-posedness]\label{T:local}
Given $(u_0, u_1)\in\dot H_x^{s_c}\times \dot H_x^{s_c-1}$ and $t_0\in \R$, there is a unique maximal-lifespan solution $u:I\times\R^3 \rightarrow \R$ to
\eqref{nlw} with initial data $(u(t_0), u_t(t_0))=(u_0, u_1)$.  This solution also has the following properties:
\begin{CI}
\item (Local existence) $I$ is an open neighbourhood of $t_0$.
\item (Blowup criterion) If $\sup I$ is finite, then $u$ blows up forward in time (in the sense of Definition~\ref{D:blowup}); if $\inf I$ is finite,
then $u$ blows up backward in time.
\item (Scattering) If $\sup I=+\infty$ and $u$ does not blow up forward in time, then $u$ scatters forward in time, that is,
there exists a unique $(u_0^+, u_1^+) \in \dot H_x^{s_c}\times \dot H_x^{s_c-1}$ such that
\begin{equation}\label{like u+}
\lim_{t \to +\infty} \Bigl\| u(t) - \cos(t|\nabla|)u_0^+ - \tfrac{\sin(t|\nabla|)}{|\nabla|}u_1^+ \Bigr\|_{\dot H^{s_c}_x} = 0.
\end{equation}
Conversely, given $(u_0^+, u_1^+) \in \dot H_x^{s_c}\times \dot H_x^{s_c-1}$ there is a unique solution to \eqref{nlw} in a neighbourhood of infinity
so that \eqref{like u+} holds.
\item (Small data global existence) If $(u_0, u_1)$ is sufficiently small in $\dot H_x^{s_c}\times \dot H_x^{s_c-1}$, then $u$ is a global solution
which does not blow up either forward or backward in time.  Indeed, in this case
$$
S_\R(u)\lesssim \bigl\|(u_0,u_1)\bigr\|_{\dot H_x^{s_c}\times \dot H_x^{s_c-1}}^{2p}.
$$
\end{CI}
\end{theorem}

Our next topic is the energy flux identity/inequality, which is a variant of the Morawetz identity/inequality discussed in
the introduction and is proved in much the same way.  We will use it in connection with the Duhamel formulae \eqref{Duhamel}, in order
to show that the time integrals converge.

\begin{lemma}[Energy flux inequality]\label{L:energy flux}
If $u$ is a solution to \eqref{nlw} with $(u,u_t)\in L_t^\infty (I; \dot H^{s_c}_x\times\dot H^{s_c-1}_x)$, then
\begin{align*}
\int_I \int_{|x-y|=|t|} &|u(t,y)|^{p+2} \, dS(y)\, dt \lesssim_u \sup_{t\in I} |t|^{1-\frac4p}
\end{align*}
uniformly for $x\in\R^3$.
\end{lemma}

\begin{proof}
The result follows by applying the Fundamental Theorem of Calculus to
$$
\mathcal{E}(t) := \int_{|x-y|\leq |t|} \tfrac12|\ntx u(t,y)|^2 + \tfrac1{p+2} |u(t,y)|^{p+2}\, dy
$$
and noting that
$$
\mathcal{E}(t) \lesssim \Bigl(\|\ntx u(t)\|_{L_x^{\frac{3p}{p+2}}}^2 +\|u(t)\|_{ L_x^{\frac{3p}2}}^{p+2}\Bigr) |t|^{1-\frac4p} \lesssim_u |t|^{1-\frac4p},
$$
by H\"older's inequality and Sobolev embedding.
\end{proof}

The small data theory shows that the $\dot H^{s_c}_x\times\dot H^{s_c-1}_x$ norm of a blowup solution must remain bounded from below.  The fact that this
norm is non-local in space reduces the efficacy of this statement.  Our next lemma gives a lower bound in a more suitable norm:

\begin{lemma}[$\ntx u$ nontriviality]\label{L:u' nontriviality} Let $u$ be a global solution that is almost periodic modulo symmetries.  Then,
\begin{equation}\label{E:u' nontrivial}
    \inf_{t\in\R} \ \int_{\R^3} |\ntx u(t,x)|^\frac{3p}{p+2} \,dx \gtrsim_u 1.
\end{equation}
\end{lemma}

\begin{proof}
First we note that by the small data theory,
\begin{equation}\label{E:u nontrivial in Hsc}
    \inf_{t\in\R} \ \bigl\|\ntx u(t,x) \bigr\|_{\dot H^{s_c-1}_x} \gtrsim_p 1,
\end{equation}
for otherwise $u$ would have finite spacetime norm in contravention of the hypotheses of this lemma.
Indeed, a solution that scatters cannot be almost periodic modulo symmetries.

Next, we note that
$$
\| f \|_{L_x^\frac{3p}{p+2}} \div \| f \|_{\dot H^{s_c-1}_x} > 0
$$
for any non-zero $\R^4$-valued $f\in\dot H^{s_c-1}_x$. Hence this ratio achieves a non-zero minimum on any compact set
that does not contain the zero function. Indeed, since this ratio is invariant under scaling and translation, it
suffices for the set to be compact modulo these symmetries.  Therefore, the ratio is bounded from below on the (precompact) orbit
$\ntx u(t)$, and so, in view of \eqref{E:u nontrivial in Hsc}, the lemma follows.  Note also that \eqref{E:u nontrivial in Hsc}
guarantees that the orbit $\ntx u(t)$ does not approach the zero function.
\end{proof}

It is not possible to obtain lower bounds on the norm of $u(t)$ for a single time $t$, as it is quite conceivable that $u(t)=0$,
with all the $\dot H^{s_c}_x \times \dot H^{s_c-1}_x$ norm concentrating in $u_t(t)$.  Nevertheless, this phenomenon must be rather
rare as our next lemma demonstrates.

\begin{lemma}[$L^{3p/2}_x$-norm nontriviality]\label{L:3p/2 nontriviality}
Let $u$ be a global solution that is almost periodic modulo symmetries.  Then, for any $A>0$, there exists $\eta=\eta(u,A)>0$ so that
\begin{equation}\label{E:positive density}
\bigl| \bigl\{ t \in [t_0, t_0 + A N(t_0)^{-1}] : \|u(t)\|_{L^{3p/2}_x(\R^3)} \geq \eta \bigr\} \bigr| \geq \eta N(t_0)^{-1}
\end{equation}
for all $t_0\in\R$.
\end{lemma}

\begin{proof}
Recasting \eqref{E:positive density} in terms of the normalizations of $u$, defined in \eqref{E:Normalization}, yields
\begin{equation}\label{E:PD2}
\bigl| \bigl\{ s \in [0,A] : \|u^{[t_0]}(s)\|_{L^{3p/2}_x(\R^3)} \geq \eta \bigr\} \bigr| \geq \eta.
\end{equation}
As the map from the initial data to the solution is continuous (a consequence of the local theory) and the set
$$
\bigl\{ \bigl( u^{[t_0]}(0),\; u^{[t_0]}_t(0),\; s\bigr) : t_0\in\R \text{ and } s\in[0,A] \bigr\}
$$
is precompact, we deduce that $\{ u^{[t_0]}(s) : t_0\in\R \text{ and } s\in[0,A]\}$ is precompact in $\dot H^{s_c}_x$.
Thus by Sobolev embedding, we see that it suffices to show that for some choice of $\eta$ the set appearing in \eqref{E:PD2} is non-empty for all $t_0\in\R$.
(Of course, the passage from non-emptyness to positive measure requires a reduction in $\eta$.)

To see that the set appearing in \eqref{E:PD2} is non-empty, we argue by contradiction.   To this end, imagine that
there is a sequence of times $t_n$ so that
\begin{equation}\label{E:PDblah}
\|u^{[t_n]}(s)\|_{L^{3p/2}_x(\R^3)} \to 0 \qtq{uniformly for}  s \in [0,A].
\end{equation}
Then, by a simple bootstrap argument using the Duhamel formula, \eqref{old duhamel}, and the Strichartz inequality, we
deduce that
$$
\bigl\| u^{[t_n]}(s) - \cos(s|\nabla|) u^{[t_n]}(0) - |\nabla|^{-1} \sin(s|\nabla|) u_s^{[t_n]}(0)
    \bigr\|_{L^\infty_s L^{3p/2}_x([0,A]\times\R^3)}
     \to 0.
$$
Thus, appealing to \eqref{E:PDblah} once again, we obtain
\begin{equation}\label{E:PDblah2}
\bigl\| \cos(s|\nabla|) u^{[t_n]}(0) + |\nabla|^{-1} \sin(s|\nabla|) u_s^{[t_n]}(0) \bigr\|_{L^\infty_s L^{3p/2}_x([0,A]\times\R^3)}
     \to 0.
\end{equation}
This, we will see, contradicts the uniqueness theorem for the linear wave equation.

As noted previously (cf. the remarks after Definition~\ref{D:ap}), almost-periodicity of $u$ implies that the sequence
of pairs $(u^{[t_n]}(0),u_s^{[t_n]}(0))$ is precompact in $\dot H^{s_c}_x\times \dot H^{s_c-1}_x$. Thus, by passing to a
subsequence, we may assume that it converges and name the limit $(f,g)$.  This limit must be non-zero, for otherwise, we
could apply the small-data theory to the pair $(u(t_n),u_t(t_n))$, for $n$ large enough, and deduce that $u$ is global
and has finite spacetime norm.

On the other hand, from \eqref{E:PDblah2} we see that
\begin{equation*}
 \cos(s|\nabla|) f + |\nabla|^{-1} \sin(s|\nabla|) g = 0 \qtq{for all} s\in[0,A],
\end{equation*}
which implies $f\equiv g\equiv 0$, since they can be reconstructed from the behaviour of this solution of the linear
wave equation as $s\to0$.  This contradicts the results of the previous paragraph and so completes the proof of the lemma.
\end{proof}

\begin{corollary}[Potential energy concentration] \label{C:pot conc}
Let $u$ be a global solution that is almost periodic modulo symmetries.
Then, there exists $C=C(u)$ so that
\begin{equation}\label{E:some pot E}
\int_I \int_{|x-x(t)|\leq C/ N(t)} |u(t,x)|^{p+2} \,dx\,dt  \gtrsim_u \int_I N(t)^{\frac4p-1} \,dt
\end{equation}
uniformly for all intervals $I=[t_1,t_2]\subseteq \R$ with $t_2\geq t_1 + N(t_1)^{-1}$.
\end{corollary}

\begin{proof}
We know that there exists $\delta=\delta(u)$ so that
\begin{equation}\label{E:N loc constant}
N(t) \sim_u N(t_0) \qtq{uniformly for} t\in[t_0-\delta N(t_0)^{-1}, t_0+\delta N(t_0)^{-1}]\qtq{and} t_0\in\R.
\end{equation}
Indeed, if it were not possible to choose a $\delta$ with this property, then one could find a convergent sequence of
initial data (taken from normalizations of $u$) whose limit blows up instantaneously, in contradiction to the local
theory.  For further details, see \cite[Corollary~3.6]{KTV} or \cite[Lemma~5.18]{ClayNotes}.  We note that
the argument requires perturbation theory, which is an ingredient in the proof of Theorem~\ref{T:reduct}.

In view of \eqref{E:N loc constant}, it suffices to prove the result for intervals of the form
$[t_0, t_0+\delta N(t_0)^{-1}]$ for some small fixed $\delta>0$.  The simple argument that shows this requires that
$I$ contain at least one interval of this form.   This is the origin of the requirement $t_2\geq t_1 + N(t_1)^{-1}$ in the statement of the corollary;
correspondingly, we require $\delta\leq1$.

As noted in \eqref{E:u' compact}, the almost-periodicity of $u$ and Sobolev embedding imply that for
any $\eta>0$ there exists $C(\eta)>0$ so that
$$
\int_{|x-x(t)|\geq C(\eta)/N(t)} |u(t,x)|^{\frac{3p}2} \,dx \leq \eta.
$$
Combining this with Lemma~\ref{L:3p/2 nontriviality} yields the following:  There exists $C=C(u)$ so that the set~of
$$
t\in[t_0,t_0+\delta N(t_0)^{-1}] \qtq{such that}
\int_{|x-x(t)|\leq C/ N(t)} |u(t,x)|^{3p/2} \,dx \gtrsim_u 1
$$
has measure $\gtrsim_u N(t_0)^{-1}$.  In view of this, it suffices to show that given $\eta_0>0$ there exists
$\eta_1=\eta_1(u,\eta_0)>0$ so that
\begin{equation*}%\label{E:crit to pot E}
\int_{|x-x(t)|\leq C/N(t)} |u(t,x)|^{3p/2} \,dx \geq\eta_0
\implies N(t)^{1-\frac4p} \int_{|x-x(t)|\leq C/N(t)} |u(t,x)|^{p+2} \,dx \geq\eta_1.
\end{equation*}
The truth of this statement follows from the almost-periodicity of $u$.  Indeed, passing to the normalizations of
$u(t)$ and recalling that these form a precompact set in $L^{3p/2}_x$, the statement reduces to the fact that if a
sequence $\{f_n\}$ converges in $L^{3p/2}_x$ and converges to zero in $L^{p+2}_x$ then it converges to zero in
$L^{3p/2}_x$.
\end{proof}

%%%%%%%%%%%%%%%%%%%%%%%%%%%%%%%%%%%%%%%%%%%%%%%%%%%%%%%%%%%%%%%%%%%%%%%%%%%%%%%%%%%%%%%%%%%
%
%
%                                   Section
%
%
%%%%%%%%%%%%%%%%%%%%%%%%%%%%%%%%%%%%%%%%%%%%%%%%%%%%%%%%%%%%%%%%%%%%%%%%%%%%%%%%%%%%%%%%%%%

\section{Global enemies are sub-luminal}\label{S:subluminal}

The principal goal of this section is to show that for the global enemies (the soliton-like and frequency-cascade solutions) of
Theorem~\ref{T:enemies}, the center $x(t)$ of the wave packet travels strictly slower than the speed of light, at least on
average, over reasonably long time intervals.  Note that Definition~\ref{D:ap} does not define $x(t)$ uniquely, but only
up to a radius of about $N(t)^{-1}$.  While this does not render the goal of this section ambiguous, it is something of
a nuisance in the proof.  For that reason, we first standardize $x(t)$ in some mild fashion.  This is our first
proposition.  The main result of the section, the sub-luminality of global enemies, is Proposition~\ref{P:subluminal}.

\begin{proposition}[Centering $x(t)$]\label{P:Standard}
Let $u$ be a global almost periodic solution to \eqref{nlw}.  The function $x(t)$ can be modified so that it retains all properties
stated in Definition~\ref{D:ap} (though $C(\eta)$ may need to be made larger) and in addition satisfies the following:
For some large constant $C_u$ and all $\omega\in S^2$,
\begin{equation}\label{x splits u}
\int_{\omega\cdot(x-x(t))>0} |\ntx u(t,x)|^{\frac{3p}{p+2}}\,dx \geq \tfrac1{C_u}
\end{equation}
that is, each plane through $x(t)$ partitions $u$ into two \emph{non-trivial} pieces.  Moreover,
\begin{equation}\label{x is Lip}
|x(t_1) - x(t_2)| \leq |t_1-t_2| + C_u N(t_1)^{-1} + C_u N(t_2)^{-1} \qtq{for any} t_1,t_2\in \R;
\end{equation}
indeed, this was also true for the original $x(t)$.
\end{proposition}

Before proceeding to the proof of this proposition, we pause to make the following intuition precise: Compactness
(modulo scaling) prevents the solution $u(t)$ from concentrating on very narrow strips, provided the width is measured in
units of $N(t)^{-1}$.

\begin{lemma} [Small on narrow strips]\label{L:strip}
Let $u$ be a global almost periodic solution to \eqref{nlw}.  Then for any $\eta>0$ there exists a small
constant $c(\eta)>0$ so that
\begin{equation}\label{E:strip}
\sup_{\omega\in S^2}\int_{|\omega\cdot(x-x(t))|<c(\eta)/N(t)} |\ntx u(t,x)|^{\frac{3p}{p+2}}\,dx \leq \eta.
\end{equation}
\end{lemma}

\begin{proof}
For a single value of $t$ this follows from the Monotone Convergence Theorem; it extends to the full orbit of $\ntx u(t)$
by compactness.
\end{proof}

Our first application of this lemma is to the proof of Proposition~\ref{P:Standard}; we will use it again in the proof
of Proposition~\ref{P:subluminal}.

\begin{proof}[Proof of Proposition~\ref{P:Standard}]
We first prove \eqref{x is Lip}.  As the veracity of this equation will be deduced from the properties stated in
Definition~\ref{D:ap}, it will be equally valid for the modified version of $x(t)$ which will be defined in due course.

Choose $\eta>0$ to be a small number well below the $\dot H^{s_c}_x\times\dot H^{s_c-1}_x$ threshold for the small data theory.
By Definition~\ref{D:ap}, there is a constant $C(\eta)$ so that
\begin{equation}\label{sub:tight}
\bigl\|  \phi\bigl(\tfrac{x-x(t_1)}{C(\eta) N(t_1)^{-1} }\bigr) u(t_1,x) \bigr\|_{\dot H^{s_c}}
    + \bigl\|  \phi\bigl(\tfrac{x-x(t_1)}{C(\eta) N(t_1)^{-1} }\bigr) u_t(t_1,x) \bigr\|_{\dot H^{s_c-1}}  \leq \eta
\end{equation}
for some smooth cutoff $\phi:\R^3\to[0,\infty)$ with $\phi(x)=1$ for $|x|\geq 1$ and $\phi(x)=0$ for $|x|\leq
\tfrac12$. Thus, by the small data theory, there is a global solution to \eqref{nlw} whose Cauchy data at time $t_1$ match the
combination of $\phi$ and $u$ given in \eqref{sub:tight}.  Moreover, per the small data theory, each critical
Strichartz norm of this solution is controlled by a ($p$-dependent) multiple of $\eta$. By simple domain of dependence
arguments, this new solution agrees with the original $u$ on the set
$$
\Omega(t):=\{x:|x-x(t_1)| \geq |t-t_1| + C(\eta)/N(t_1) \} \qtq{for all} t\in \R
$$
and hence, by Sobolev embedding,
\begin{align*}
\bigl\| \ntx u(t) \bigr\|_{L^{\frac{3p}{p+2}}_x(\Omega(t))} \lesssim \eta \qtq{for all} t\in \R.
\end{align*}
Now consider this estimate and \eqref{E:u' compact} with $t=t_2$ and $\eta$ much less than half the
minimal $L^{3p/(p+2)}_x$ norm of $\ntx u(t)$, over time; this minimum is positive by virtue of Lemma~\ref{L:u' nontriviality}. Thus
we may deduce that
$$
\{ x : |x-x(t_1)| \leq |t_2-t_1| + C(\eta)/N(t_1) \} \cap  \{ x : |x-x(t_2)| \leq C(\eta)/N(t_2) \} \neq \varnothing,
$$
from which \eqref{x is Lip} follows.

We now turn to the proof of \eqref{x splits u}.  First, fix $C>0$ so that $B(t):=\{|x-x(t)| \leq C/N(t)\}$ obeys
\begin{equation}\label{sub:standardstar}
\int_{B(t)} \bigl| \ntx u(t,x)\bigr|^{\frac{3p}{p+2}} \, dx \gtrsim_u 1
\end{equation}
uniformly for $t\in\R$. This is possible by virtue of Lemma~\ref{L:u' nontriviality} and \eqref{E:u' compact}.  Now set
$$
\tilde x(t) := x(t) + \frac{\int_{B(t)} [x-x(t)]  \, | \ntx u(t,x) |^{\frac{3p}{p+2}} \, dx}{\int_{B(t)} | \ntx u(t,x) |^{\frac{3p}{p+2}} \, dx}.
$$
This definition immediately implies $|\tilde x(t) - x(t)| \leq C/N(t)$; thus, $\tilde x(t)$ maintains the properties
stated in Definition~\ref{D:ap}, though the compactness modulus function $C(\eta)$ may need to be increased, say by the
addition of $C$.  In particular, \eqref{x is Lip} remains valid after suitable increase in the constant $C_u$.

By construction,
$$
\int_{B(t)} \omega\cdot [x-\tilde x(t)] \, \bigl|\ntx u(t,x)\bigr|^{\frac{3p}{p+2}} \, dx =0
$$
for any (unit) vector $\omega\in S^2$, while by \eqref{sub:standardstar} and Lemma~\ref{L:strip},
$$
\int_{B(t)} \bigl|\omega\cdot [x-\tilde x(t)]\bigr| \, \bigl|\ntx u(t,x)\bigr|^{\frac{3p}{p+2}} \, dx  \gtrsim_u N(t)^{-1}.
$$
Putting these two results together  yields
$$
\int_{B(t)} \bigl\{\omega\cdot [x-\tilde x(t)]\bigr\}_+ \, \bigl|\ntx u(t,x)\bigr|^{\frac{3p}{p+2}} \, dx  \gtrsim_u N(t)^{-1},
    \qtq{where} \{y\}_+ = \max\{0,y\}.
$$
Therefore, as $x\in B(t)$ implies $|x-\tilde x(t)|\leq 2C N(t)^{-1}$, we have
$$
\int_{\omega\cdot(x-\tilde x(t))>0} \, \bigl|\ntx u(t,x)\bigr|^{\frac{3p}{p+2}}\,dx
    \geq \int_{B(t)} \frac{\{\omega\cdot [x-\tilde x(t)]\}_+}{2CN(t)^{-1}} \, \bigl|\ntx u(t,x)\bigr|^{\frac{3p}{p+2}}\,dx \gtrsim_u 1,
$$
which proves \eqref{x splits u}.
\end{proof}

\begin{proposition}[Global enemies are sub-luminal]\label{P:subluminal}
Let $u$ be a global almost periodic solution to \eqref{nlw} with $N(t)\geq 1$.  Then there exists $\delta=\delta(u)>0$ such that
\begin{align}\label{E:subluminal}
|x(t)-x(\tau)|\leq (1-\delta) |t-\tau| \quad \text{whenever} \quad |t-\tau|\geq \tfrac1\delta.
\end{align}
\end{proposition}

The proof of this proposition splits into two cases depending on whether or not $N(t)$ varies significantly over the
time interval between $t$ and $\tau$.   Before turning to the main part of the proof of Proposition~\ref{P:subluminal},
we present the key ingredient in the case of significant variation as a lemma:

\begin{lemma}\label{L:live fast}
For almost periodic solutions $u$ to \eqref{nlw}, there exists $c=c(u)>0$ so that
\begin{equation}\label{E:live fast}
 |x(t_1) - x(t_2)| \geq |t_1-t_2| - c N(t_1)^{-1} \quad\implies\quad N(t_2) \leq c^{-2} N(t_1).
\end{equation}
For a non-vacuous statement, we assign the names $t_1$ and $t_2$ so that $N(t_1)\leq N(t_2)$.
\end{lemma}

\begin{proof}
By time-reversal symmetry, we may assume that $t_1<t_2$.  By space translation symmetry, we set $x(t_1)=0$ and by
rotation symmetry, we assume $x(t_2)=(x_1(t_2),0,0)$ with $x_1(t_2)\geq 0$.

Assume, toward a contradiction, that $cN(t_1)^{-1} \geq c^{-1} N(t_2)^{-1}$.  Then, by choosing $c$ small enough (depending on
$\eta$) and invoking the almost periodicity of $u$, we obtain
\begin{equation}\label{sub:tight again}
\bigl\| \psi\bigl( \tfrac{x_1-x_1(t_2)}{cN(t_1)^{-1}}\bigr) u(t_2,x) \bigr\|_{\dot H^{s_c}_x}
    + \bigl\| \psi\bigl( \tfrac{x_1-x_1(t_2)}{cN(t_1)^{-1}}\bigr) u_t (t_2,x) \bigr\|_{\dot H^{s_c-1}_x}  \leq \eta
\end{equation}
for some smooth cutoff $\psi:\R\to[0,\infty)$ with $\psi(x)=1$ for $x\leq -1$ and $\psi(x)=0$ for $x\geq -1/2$. Here
$\eta$ is chosen below the threshold for the small data theory.  Using this theory and simple domain of dependence
arguments, we may deduce that
\begin{equation}\label{sub:back at t1}
\int_{\Omega} \bigl| \ntx u(t_1,x) \bigr|^\frac{3p}{p+2} \,dx \lesssim \eta^\frac{3p}{p+2}
    \qtq{where} \Omega:=\bigl\{ x: x_1 \leq x_1(t_2) - (t_2-t_1) - cN(t_1)^{-1} \bigr\}.
\end{equation}
Now by LHS\eqref{E:live fast} and the standardizations introduced at the beginning of this proof,
$$
\Omega \supseteq \bigl\{ x: (- e_1) \cdot (x - x(t_1)) \geq 2cN(t_1)^{-1} \bigr\},
$$
with the obvious consequence for the $L^{3p/(p+2)}_x$ norm of $\ntx u(t_1)$ on this set.  Making $\eta$ small enough and
then $c$ small enough, we deduce a contradiction to the combination of Lemma~\ref{L:strip} and \eqref{x splits u}.
\end{proof}

\begin{proof}[Proof of Proposition~\ref{P:subluminal}]
We claim that it suffices to show that there exists $A=A(u)>1$ so that for all $t_0\in\R$ there exists
$t\in[t_0,t_0+AN(t_0)^{-1}]$ so that
\begin{equation}\label{E:sub star}
|x(t) - x(t_0)| \leq |t-t_0| - A^{-1} N(t_0)^{-1}.
\end{equation}
Indeed, with this claim in hand, we may inductively construct a sequence of times $\{t_k\}$ so that $t_0=0$,
$0<t_{k+1}-t_k \leq AN(t_k)^{-1}$, and
\begin{align*}
|x(t_m) - x(t_l)| &\leq \sum_{k=l}^{m-1} |t_{k+1}-t_k| - A^{-1} N(t_k)^{-1} \\
& \leq \sum_{k=l}^{m-1} (1 - A^{-2}) |t_{k+1}-t_k|
    \leq (1-A^{-2}) |t_m-t_l|.
\end{align*}
We may deduce the result for values of $t$ and $\tau$ lying in $[0,\infty)$ between these sample points by applying
\eqref{x is Lip}.  Note that this requires choosing $\delta\leq \frac12 A^{-2} (A+2C_u)^{-1}$, where $C_u$ is as in \eqref{x is Lip}. By
employing time-reversal symmetry, one similarly obtains the result for $t,\tau\in(-\infty,0]$ and thence for all pairs
of times via the triangle inequality.

We now turn to verifying the claim made at the beginning of this proof.  Let $c$ be as in Lemma~\ref{L:live fast}.  If
$N(t) > c^{-2} N(t_0)$ for some $t\in[t_0,t_0+AN(t_0)^{-1}]$, then by that lemma,
$$
|x(t) - x(t_0)| \leq |t-t_0| - cN(t_0)^{-1} \leq |t-t_0| - A^{-1} N(t_0)^{-1},
$$
provided we ensure $A\geq c^{-1}$.  This settles this case.  Suppose now that $N(t) < c^{2} N(t_0)$ for some
$t\in[t_0,t_0+AN(t_0)^{-1}]$.  Then by Lemma~\ref{L:live fast},
$$
|x(t) - x(t_0)| \leq |t-t_0| - cN(t)^{-1} \leq |t-t_0| - c^{-1}N(t_0)^{-1} \leq |t-t_0| - A^{-1} N(t_0)^{-1},
$$
provided $A$ is chosen so that $A\geq c$.  This settles this case.

It remains to verify our claim in the case
\begin{equation}\label{E:sub N const}
c^2 \leq \frac{N(t)}{N(t_0)} \leq c^{-2} \qtq{for all} t\in[t_0,t_0+AN(t_0)^{-1}],
\end{equation}
for which we will argue by contradiction.  For notational convenience, we translate so that $t_0=0$ and $x(t_0)=0$.  Now by assuming that \eqref{E:sub star}
fails and making use of \eqref{x is Lip} and \eqref{E:sub N const}, we deduce that
$$
 t\in[0, AN(0)^{-1}] \implies \bigl| |x(t)| - t \bigr| \leq B N(0)^{-1}
$$
for some $B=B(u)\geq C_u (1+c^{-2}) +1$, where $C_u$ is as in \eqref{x is Lip}.  By enlarging $B$ and using \eqref{E:sub N const}, we can ensure that
$$
\bigl\{\bigl| |x|-t\bigr|\leq B/N(0)\bigr\}\supseteq \{ |x-x(t)|\leq C/N(t)\bigr\},
$$
with $C$ as in Corollary~\ref{C:pot conc}.  Using this corollary, it follows that
\begin{align}\label{comp}
\int_{B/N(0)}^{A/N(0)} \int_{||x|-t| \leq B/N(0)} |u(t,x)|^{p+2}\, dx\,dt \gtrsim_u (A-B) N(0)^{\frac4p - 2},
\end{align}
whenever $A\geq B + c^{-2}$.

On the other hand, we can obtain an upper bound on LHS\eqref{comp} from the energy flux identity.  As a first step, we observe that by
Lemma~\ref{L:energy flux} we have
\begin{align*}
\int_{\R^3} \chi_{\{|x|\leq 2B/N(0)\}} \int_{B/N(0)}^{A/N(0)} \int_{|x-y|=t} |u(t,y)|^{p+2} \,dS(y)\,dt\,dx \lesssim_u A^{1-\frac4p} B^3 N(0)^{\frac4p - 4}.
\end{align*}
To continue, we change variables via $y=x+z$ and then $x=x'-z$ to obtain
\begin{align*}
\int_{B/N(0)}^{A/N(0)} \int_{\R^3}\int_{|z|=t} |u(t,x')|^{p+2} \chi_{\{|x'-z|\leq 2B/N(0)\}} \,dS(z)\,dx'\,dt
    \lesssim_u A^{1-\frac4p} B^3 N(0)^{\frac4p - 4}.
\end{align*}
Noting that
$$
\int_{|z|=t} \chi_{\{|x'-z|\leq 2L\}} \,dS(z) \gtrsim L^2 \qtq{when} \bigl||x'|-t\bigr|\leq L \qtq{and} |t|\geq L
$$
for any $L>0$ and hence for $L=B/N(0)$, we are lead to
\begin{align*}
\int_{B/N(0)}^{A/N(0)} \int_{||x'|-t| \leq B/N(0)} |u(t,x')|^{p+2} \,dx'\,dt \lesssim_u A^{1-\frac4p} B N(0)^{\frac4p - 2}.
\end{align*}
To finish the proof, we merely note that this contradicts \eqref{comp} once $A$ is chosen sufficiently large.
\end{proof}

%%%%%%%%%%%%%%%%%%%%%%%%%%%%%%%%%%%%%%%%%%%%%%%%%%%%%%%%%%%%%%%%%%%%%%%%%%%%%%%%%%%%%%%%%%%
%
%
%                                   Section
%
%
%%%%%%%%%%%%%%%%%%%%%%%%%%%%%%%%%%%%%%%%%%%%%%%%%%%%%%%%%%%%%%%%%%%%%%%%%%%%%%%%%%%%%%%%%%%

\section{Additional decay}\label{S:AddnlDecay}

In this section we prove additional decay for the soliton-like and frequency-cascade solutions described in Theorem~\ref{T:enemies}.

\begin{proposition}[$L^q$ breach of scaling]\label{P:L^p breach}
Let $u$ be a global solution to \eqref{nlw} that is almost periodic modulo symmetries.  In particular,
\begin{align}\label{Hs bounded}
\bigl\|(u,u_t)\bigr\|_{L_t^\infty(\R; \dot H^{s_c}_x\times\dot H_x^{s_c-1})}<\infty.
\end{align}
Assume also that
\begin{align}\label{inf N bounded}
\inf_{t\in \R} N(t)\geq 1.
\end{align}
Then $u\in L_t^\infty L_x^q$ for $\frac{3p^2+20p-16}{6p}<q\leq \frac{3p}2$.  In particular, $u\in L_t^\infty L_x^p$
(as $p\geq 6$) and by H\"older's inequality, $F(u) \in L_t^\infty L_x^1$.
\end{proposition}

The remainder of this section is dedicated to the proof of Proposition~\ref{P:L^p breach}.

Let $u$ be a solution to \eqref{nlw} that obeys the hypotheses of Proposition~\ref{P:L^p breach}.  Let $\eta>0$ be a small constant to be
chosen later.  Then by almost periodicity modulo symmetries combined with \eqref{inf N bounded}, there exists $N_0=N_0(\eta)$ such that
\begin{align}\label{Hs small}
\bigl\||\nabla|^{s_c} P_{\leq N_0}u \bigr\|_{L_t^\infty L_x^2(\R\times\R^3)} + \bigl\||\nabla|^{s_c-1} P_{\leq N_0}u_t \bigr\|_{L_t^\infty L_x^2(\R\times
\R^3)}\leq \eta.
\end{align}

Now for $\frac{3p}2< r<\infty$ define
\begin{equation*}
A_r(N):= N^{\frac3r-\frac2p} \sup_{t\in \R} \|u_N(t)\|_{L_x^r}
\end{equation*}
for frequencies $N\leq 10p N_0$.  Note that by Bernstein's inequality combined with Sobolev embedding and \eqref{Hs bounded},
\begin{align}\label{A_r bdd}
A_r(N)\lesssim  \|u_N\|_{L_t^\infty L_x^{\frac{3p}2}} \lesssim \bigl\||\nabla|^{s_c} u\bigr\|_{L_t^\infty L_x^2} <\infty
\end{align}
for all $N\leq 10p N_0$.

We next prove a recurrence formula for $A_r(N$).

\begin{lemma}[Recurrence]\label{L:recurrence}
For $\frac{3p}2< r<\infty$ we have
\begin{align}\label{E:recurrence}
A_r(N)\lesssim_u \Bigl\{\bigl(\tfrac{N}{N_0}\bigr)^{1-\frac2p-\frac3r}
&+ \eta^2 \sum_{\frac{N}{10p}\leq M\leq N_0} \bigl(\tfrac{N}{M}\bigr)^{1-\frac2p-\frac3r-}A_r(M)^{p-1}\notag\\
&+ \eta^2 \sum_{M<\frac{N}{10p}} \bigl(\tfrac{M}{N}\bigr)^{\frac12+\frac3r-}A_r(M)^{p-1}\Bigr\}^{\frac{4r-6p}{r(p+4)}},
\end{align}
for all $N\leq 10p N_0$.
\end{lemma}

\begin{proof}
Fix $N\leq 10p N_0$.  By time-translation symmetry, it suffices to prove
\begin{align}\label{rec goal}
N^{\frac3r-\frac2p}\| u_N(0)\|_{L_x^r}\lesssim_u \Bigl\{\bigl(\tfrac{N}{N_0}\bigr)^{1-\frac2p-\frac3r}
&+ \eta^2 \sum_{\frac{N}{10p}\leq M\leq N_0} \bigl(\tfrac{N}{M}\bigr)^{1-\frac2p-\frac3r-}A_r(M)^{p-1}\notag\\
&+ \eta^2 \sum_{M<\frac{N}{10p}} \bigl(\tfrac{M}{N}\bigr)^{\frac12+\frac3r-}A_r(M)^{p-1}\Bigr\}^{\frac{4r-6p}{r(p+4)}}.
\end{align}

Using the Duhamel formula \eqref{Duhamel} into the future we write
\begin{equation*}
u_N(0)= \int_{0}^{\infty} -\frac{\sin\bigl(t\sqrtDelta\bigr)}{\sqrtDelta}F_N(u(t))\,dt.
\end{equation*}
Now let $T>0$ to be chosen later.  Using the explicit form of the propagator (cf. Lemma~\ref{L:propagator}), H\"older's
inequality, and the energy flux inequality Lemma~\ref{L:energy flux}, we estimate the long-time contribution (without
the Littlewood--Paley projection) as follows:
\begin{align}
\Bigl\|\int_{T}^{\infty} -\frac{\sin\bigl(t\sqrtDelta\bigr)}{\sqrtDelta} F(u(t))\,dt\Bigr\|_{L_x^\infty}
&\lesssim \Bigl\| \int_{T}^{\infty} \frac1t\int_{|x-y|=t} F(u(t,y)) \, dS(y)\, dt \Bigr\|_{L_x^\infty} \notag\\
&\lesssim \sum_{R\geq T} \frac1R \Bigl\| \int_{R}^{2R} \int_{|x-y|=t} F(u(t,y)) \, dS(y)\, dt \Bigr\|_{L_x^\infty} \notag\\
&\lesssim \sum_{R\geq T} \frac1R R^{\frac3{p+2}}\Bigl\| \int_{R}^{2R} \int_{|x-y|=t} |u(t,y)|^{p+2} \, dS(y)\, dt \Bigr\|_{L_x^\infty}^{\frac{p+1}{p+2}}
\notag\\
&\lesssim_u \sum_{R\geq T} \frac1R R^{\frac3{p+2}} R^{(1-\frac4p)\frac{p+1}{p+2}} \notag\\
&\lesssim_u T^{-\frac2p}. \label{E:LongTimeEstimate}
\end{align}
On the other hand, by \eqref{Duhamel},
$$
\int_{T}^{\infty} -\frac{\sin\bigl(t\sqrtDelta\bigr)}{\sqrtDelta} F(u(t))\,dt
= \cos\bigl(T\sqrtDelta\bigr)u(T) - \frac{\sin\bigl(T\sqrtDelta\bigr)}{\sqrtDelta}u_t(T),
$$
and so, using Sobolev embedding and \eqref{Hs bounded} we get
$$
\Bigl\|\int_{T}^{\infty} -\frac{\sin\bigl(t\sqrtDelta\bigr)}{\sqrtDelta} F(u(t))\,dt\Bigr\|_{L_x^{\frac{3p}2}}
\lesssim \bigl\||\nabla|^{s_c} u\bigr\|_{L_t^\infty L_x^2} + \bigl\||\nabla|^{s_c-1} u_t\bigr\|_{L_t^\infty L_x^2} \lesssim_u 1.
$$
Therefore, interpolation (and $L^r$-boundedness of the Littlewood--Paley projection) yields the following estimate for
the long-time contribution:
\begin{align}\label{long ime}
\Bigl\|\int_{T}^{\infty} -\frac{\sin\bigl(t\sqrtDelta\bigr)}{\sqrtDelta}& F_N(u(t))\,dt\Bigr\|_{L_x^r}
\lesssim \Bigl\|\int_{T}^{\infty} -\frac{\sin\bigl(t\sqrtDelta\bigr)}{\sqrtDelta} F(u(t))\,dt\Bigr\|_{L_x^r}
    \lesssim_u T^{\frac3r-\frac2p},
\end{align}
valid for any $\frac{3p}2\leq r \leq \infty$.  We will make use of this inequality without the frequency projection in
the next section; see \eqref{quant:long time}.

We turn now to the short-time contribution.  By the Bernstein and Strichartz inequalities,
\begin{align}\label{short time}
\Bigl\|\int_0^{T} -\frac{\sin\bigl(t\sqrtDelta\bigr)}{\sqrtDelta} F_N(u(t))\,dt\Bigr\|_{L_x^r}
&\lesssim N^{\frac32-\frac3r}\Bigl\|\int_0^{T} -\frac{\sin\bigl(t\sqrtDelta\bigr)}{\sqrtDelta} F_N(u(t))\,dt\Bigr\|_{L_x^2}\notag\\
&\lesssim N^{\frac32-\frac3r} \|F_N(u)\|_{L_t^{\frac{2r}{r+6}}L_x^{\frac{r}{r-1}}([0,T]\times\R^3)}\notag\\
&\lesssim N^{\frac32-\frac3r}T^{\frac{r+6}{2r}} \|F_N(u)\|_{L_t^\infty L_x^{\frac{r}{r-1}}}.
\end{align}

Collecting \eqref{long ime} and \eqref{short time} we obtain
\begin{align}\label{rec reduct}
N^{\frac3r-\frac2p}\| u_N(0)\|_{L_x^r}
\lesssim_u (NT)^{\frac3r-\frac2p} + (NT)^{\frac12+\frac3r} N^{1-\frac2p-\frac3r} \|F_N(u)\|_{L_t^\infty L_x^{\frac{r}{r-1}}}.
\end{align}

To estimate the right-hand side of the inequality above we decompose
$$u=u_{>N_0} + u_{\leq N_0} = u_{>N_0} + u_{\frac N{10p}\leq \cdot\leq N_0} + u_{<\frac N{10p}}$$
and thus, taking advantage of our assumption that the power $p$ is even, we write
\begin{align}\label{decomp}
F_N(u) =P_N\biggl(u_{>N_0}\sum_{k=0}^p \tbinom{p+1}{k+1} u_{>N_0}^k u_{\leq N_0}^{p-k} +
    \sum_{k=0}^p \tbinom{p+1}k u_{<\frac N{10p}}^k u_{\frac N{10p}\leq \cdot\leq N_0}^{p+1-k}\biggr).
\end{align}

To estimate the contribution of the first term in this decomposition to \eqref{rec reduct} we use H\"older, Sobolev embedding, Bernstein,
and \eqref{Hs bounded}:
\begin{align*}
\sum_{k=0}^p \tbinom{p+1}{k+1} \Bigl\|P_N\bigl( u_{>N_0}  u_{>N_0}^k u_{\leq N_0}^{p-k}\bigr)\Bigr\|_{L_t^\infty L_x^{\frac{r}{r-1}}}
&\lesssim \|u\|_{L_t^\infty L_x^{\frac{3p}2}}^p \|u_{>N_0}\|_{L_t^\infty L_x^{\frac{3r}{r-3}}}\\
&\lesssim \|u\|_{L_t^\infty \dot H^{s_c}_x}^p \bigl\| |\nabla|^{\frac12+\frac3r}u_{>N_0}\bigr\|_{L_t^\infty L_x^2}\\
&\lesssim_u N_0^{-1+\frac2p+\frac3r} \|u_{>N_0}\|_{L_t^\infty \dot H^{s_c}_x} \\
&\lesssim_u N_0^{-1+\frac2p+\frac3r}.
\end{align*}

To estimate the contribution of the second term on the right-hand side of \eqref{decomp} to \eqref{rec reduct} we first note that
\begin{align*}
\Bigl\|P_N \bigl(\sum_{k=0}^p \tbinom{p+1}k &  u_{<\frac N{10p}}^k u_{\frac N{10p}\leq \cdot\leq N_0}^{p+1-k} \bigr)\Bigr\|_{L_t^\infty L_x^{\frac{r}{r-1}}} \\
&\lesssim \bigl\| u_{\frac N{10p}\leq \cdot\leq N_0}^{p+1} \bigr\|_{L_t^\infty L_x^{\frac{r}{r-1}}}
+ \bigl\|u_{<\frac N{10p}}^p u_{\frac N{10p}\leq \cdot\leq N_0}\bigr\|_{L_t^\infty L_x^{\frac{r}{r-1}}}.
\end{align*}
By H\"older, Bernstein, and \eqref{Hs small} we estimate
\begin{align*}
\bigl\| &u_{\frac N{10p}\leq \cdot\leq N_0}^{p+1} \bigr\|_{L_t^\infty L_x^{\frac{r}{r-1}}}\\
&\lesssim \!\!\sum_{\frac N{10p}\leq N_1\leq \cdots\leq N_{p+1}\leq N_0} \!\!\|u_{N_1}\|_{L_t^\infty L_x^r}\cdots\|u_{N_{p-1}}\|_{L_t^\infty L_x^r}
    \|u_{N_p}\|_{L_t^\infty L_x^{\frac{2r}{r-p}}}\|u_{N_{p+1}}\|_{L_t^\infty L_x^{\frac{2r}{r-p}}}\\
&\lesssim \!\!\sum_{\frac N{10p}\leq N_1\leq \cdots\leq N_{p+1}\leq N_0} \!\! \|u_{N_1}\|_{L_t^\infty L_x^r}\cdots\|u_{N_{p-1}}\|_{L_t^\infty L_x^r}
    N_p^{\frac2p-\frac{3(r-p)}{2r}}N_{p+1}^{\frac2p-\frac{3(r-p)}{2r}} \|u_{\leq N_0}\|_{L_t^\infty \dot H^{s_c}_x}^2\\
&\lesssim \eta^2 \sum_{\frac N{10p}\leq N_1\leq \cdots\leq N_{p-1}\leq N_0} A_r(N_1)\cdots A_r(N_{p-1})N_{p-1}^{-1+\frac2p+\frac3r}\\
&\lesssim \eta^2 N^{-1+\frac2p+\frac3r} \sum_{\frac N{10p}\leq N_1\leq \cdots\leq N_{p-1}\leq N_0} \Bigl(\frac N{N_{p-1}}\Bigr)^{1-\frac2p-
\frac3r}\bigl[A_r(N_1)^{p-1} + \cdots +A_r(N_{p-1})^{p-1}\bigr]\\
&\lesssim \eta^2 N^{-1+\frac2p+\frac3r} \sum_{\frac N{10p}\leq M\leq N_0} \Bigl(\frac N{M}\Bigr)^{1-\frac2p-\frac3r-}A_r(M)^{p-1}.
\end{align*}
Similarly, we estimate
\begin{align*}
\bigl\|&u_{<\frac N{10p}}^p u_{\frac N{10p}\leq \cdot\leq N_0}\bigr\|_{L_t^\infty L_x^{\frac{r}{r-1}}}\\
&\lesssim \|u_{\frac N{10p}\leq \cdot\leq N_0}\|_{L_t^\infty L_x^2} \sum_{N_1\leq \cdots \leq N_p<\frac{N}{10p}}
    \|u_{N_1}\|_{L_{t,x}^\infty}\cdots \|u_{N_{p-1}}\|_{L_{t,x}^\infty} \|u_{N_p}\|_{L_t^\infty L_x^{\frac{2r}{r-2}}}\\
&\lesssim_u \eta^2 N^{-\frac32+\frac2p} \sum_{N_1\leq \cdots \leq N_p<\frac{N}{10p}}N_1^{\frac3r}\|u_{N_1}\|_{L_t^\infty L_x^r}\cdots
    N_{p-1}^{\frac3r}\|u_{N_{p-1}}\|_{L_t^\infty L_x^r} N_p^{-\frac32+\frac2p+\frac3r}\\
&\lesssim_u \eta^2 N^{-\frac32+\frac2p} \sum_{N_1\leq \cdots \leq N_p<\frac{N}{10p}}N_1^{\frac2p}
    A_r(N_1)\cdots N_{p-1}^{\frac2p} A_r(N_{p-1})N_p^{-\frac32+\frac2p+\frac3r}\\
&\lesssim_u \eta^2 N^{-1+\frac2p+\frac3r} \sum_{N_1\leq \cdots \leq N_{p-1}<\frac{N}{10p}}
    \Bigl(\frac{N_1}{N}\Bigr)^{\frac2p-\frac1{p-1}(\frac32-\frac2p-\frac3r)} A_r(N_1)\times\cdots\\
&\qquad \qquad \qquad \qquad \qquad \qquad \qquad \qquad\cdots\times\Bigl(\frac{N_{p-1}}{N}\Bigr)^{\frac2p-\frac1{p-1}(\frac32-\frac2p-\frac3r)}
A_r(N_{p-1})\\
&\lesssim_u \eta^2 N^{-1+\frac2p+\frac3r}\sum_{M<\frac{N}{10p}}\Bigl(\frac{M}{N}\Bigr)^{\frac12+\frac3r-} A_r(M)^{p-1}.
\end{align*}

Putting everything together, we obtain
\begin{align*}
N^{\frac3r-\frac2p}\| u_N(0)\|_{L_x^r}
\lesssim_u (NT)^{\frac3r-\frac2p} + (NT)^{\frac12+\frac3r}\Bigl\{&\bigl(\tfrac{N}{N_0}\bigr)^{1-\frac2p-\frac3r}\\
&+ \eta^2 \!\!\!\!\!\sum_{\frac{N}{10p}\leq M\leq N_0} \bigl(\tfrac{N}{M}\bigr)^{1-\frac2p-\frac3r-}A_r(M)^{p-1}\notag\\
&+ \eta^2 \!\!\!\sum_{M<\frac{N}{10p}} \bigl(\tfrac{M}{N}\bigr)^{\frac12+\frac3r-}A_r(M)^{p-1}\Bigr\}.
\end{align*}
Setting
\begin{align*}
T:=N^{-1}\Bigl\{\bigl(\tfrac{N}{N_0}\bigr)^{1-\frac2p-\frac3r}
&+ \eta^2 \sum_{\frac{N}{10p}\leq M\leq N_0} \bigl(\tfrac{N}{M}\bigr)^{1-\frac2p-\frac3r-}A_r(M)^{p-1}\notag\\
&+ \eta^2 \sum_{M<\frac{N}{10p}} \bigl(\tfrac{M}{N}\bigr)^{\frac12+\frac3r-}A_r(M)^{p-1}\Bigr\}^{-\frac{2p}{p+4}},
\end{align*}
we deduce \eqref{rec goal}.  This completes the proof of the lemma.
\end{proof}

To resolve the recurrence in Lemma~\ref{L:recurrence} and so prove Proposition~\ref{P:L^p breach}, we need the following simple lemma:

\begin{lemma}[Acausal Gronwall inequality]\label{L:Gronwall}
Given $\eta, C,\gamma,\gamma'>0$, let $\{x_k\}_{k\geq 0}$ be a bounded non-negative sequence obeying
\begin{align}\label{Gron rec}
x_k \leq C 2^{-\gamma k} + \eta \sum_{l<k} 2^{-\gamma|k-l|} x_l + \eta \sum_{l \geq k} 2^{-\gamma'|k-l|} x_l  \qtq{for all} k\geq 0.
\end{align}
If $\eta\leq\tfrac14\min\{1-2^{-\gamma},1-2^{-\gamma'},1-2^{\rho-\gamma}\}$ for some $0<\rho<\gamma$, then $x_k \leq (4C+\|x\|_{\ell^\infty}) \, 2^{-\rho k}$.
\end{lemma}

\begin{proof}
Let $X_k:=\sup\{ x_m : m\geq k \}$, so that \eqref{Gron rec} implies
\begin{align*}
X_k &\leq C \, 2^{-\gamma k} + \eta \sum_{l<k} 2^{-\gamma|k-l|} X_l + \eta \sum_{p=0}^\infty \bigl[ 2^{-\gamma p} + 2^{-\gamma' p} \bigr] X_k \\
    &\leq C\, 2^{-\gamma k} + \eta \sum_{l<k} 2^{-\gamma|k-l|} X_l + \tfrac12 X_k.
\end{align*}
The result now follows by a simple inductive argument.
\end{proof}

Using this lemma we can now complete the

\begin{proof}[Proof of Proposition~\ref{P:L^p breach}]
For any positive $r$, the power appearing outside the braces in \eqref{E:recurrence} is less than one.  Thus, by concavity, Lemma~\ref{L:recurrence} implies
\begin{align*}
A_r(N)\lesssim_u \bigl(\tfrac{N}{N_0}\bigr)^{\gamma}
&+ \eta^2 \!\!\! \sum_{\frac{N}{10p}\leq M\leq N_0}\!\! \bigl(\tfrac{N}{M}\bigr)^\gamma A_r(M)^{(p-1)\frac{4r-6p}{r(p+4)}}
+ \eta^2 \!\!\! \sum_{M<\frac{N}{10p}}\!\! \bigl(\tfrac{M}{N}\bigr)^{\gamma'} \! A_r(M)^{(p-1)\frac{4r-6p}{r(p+4)}},
\end{align*}
for all
$$
N\leq 10p N_0, \quad \tfrac {3p}2<r<\infty, \quad \gamma < (1-\tfrac2p-\tfrac3r)\tfrac{4r-6p}{r(p+4)},
\qtq{and} \gamma' < (\tfrac12+\tfrac3r)\tfrac{4r-6p}{r(p+4)}.
$$
When $\frac{6p(p-1)}{3p-8}\leq r<\infty$, the power atop $A_r(M)$ on the right-hand side of the inequality above is $\geq 1$.  Discarding
surplus powers by invoking \eqref{A_r bdd}, we can apply Lemma~\ref{L:Gronwall} and deduce
\begin{align}\label{breach 1}
\|u_N\|_{L_t^\infty L_x^r}\lesssim_u N^{\frac2p-\frac3r+(1-\frac2p-\frac3r)\frac{4r-6p}{r(p+4)}-}
\end{align}
for all $N\leq 10p N_0$.  (In applying Lemma~\ref{L:Gronwall}, we set $N=10p \cdot 2^{-k} N_0$, $x_k=A_r(10p \cdot 2^{-k} N_0)$, and take $\eta$ sufficiently small.)

To continue, we use interpolation followed by \eqref{breach 1}, Bernstein, and \eqref{Hs bounded}:
\begin{align*}
\|u_N\|_{L_t^\infty L_x^q}
&\leq \|u_N\|_{L_t^\infty L_x^r}^{\frac{r(q-2)}{q(r-2)}} \|u_N\|_{L_t^\infty L_x^2}^{\frac{2(r-q)}{q(r-2)}}\\
&\lesssim_u N^{\frac{r(q-2)}{q(r-2)}[\frac2p-\frac3r+(1-\frac2p-\frac3r)\frac{4r-6p}{r(p+4)}]-} N^{-\frac{2(r-q)}{q(r-2)}(\frac32-\frac2p)}
\end{align*}
for all $N\leq 10p N_0$.  Thus, letting $r\to\infty$, we get
$$
\|u_N\|_{L_t^\infty L_x^q}\lesssim_u N^{\frac{6(q-2)}{q(p+4)}-\frac{3p-4}{pq}-}
$$
for all $N\leq 10p N_0$.  Therefore, using Bernstein together with \eqref{Hs bounded}, for $\frac{3p^2+20p-16}{6p}<q< \frac{3p}2$ we obtain
\begin{align*}
\|u\|_{L_t^\infty L_x^q}
&\leq \|u_{\leq N_0}\|_{L_t^\infty L_x^q} + \|u_{> N_0}\|_{L_t^\infty L_x^q}\\
&\lesssim_u \sum_{N\leq N_0} N^{\frac{6(q-2)}{q(p+4)}-\frac{3p-4}{pq}-} + \sum_{N>N_0} N^{\frac2p-\frac3q} \lesssim_{u} 1,
\end{align*}
which completes the proof of Proposition~\ref{P:L^p breach}.
\end{proof}

%%%%%%%%%%%%%%%%%%%%%%%%%%%%%%%%%%%%%%%%%%%%%%%%%%%%%%%%%%%%%%%%%%%%%%%%%%%%%%%%%%%%%%%%%%%
%
%
%                                   Section
%
%
%%%%%%%%%%%%%%%%%%%%%%%%%%%%%%%%%%%%%%%%%%%%%%%%%%%%%%%%%%%%%%%%%%%%%%%%%%%%%%%%%%%%%%%%%%%

\section{Quantitative decay}\label{S:quant}

In this section we consider the soliton-like and frequency-cascade solutions (in the sense of Theorem~\ref{T:enemies})
and obtain a quantitative bound for how such solutions decay away from $x(t)$ in a critical space, specifically, $L^{3p/2}_x$.
Note that compactness merely gives a non-quantitative decay.

\begin{proposition}[Spatial decay]\label{P:Spatial decay}
Let $u$ be a global solution to \eqref{nlw} that is almost periodic modulo symmetries.  Assume also that
\begin{align}\label{quant:input}
\bigl\|(u,u_t)\bigr\|_{L_t^\infty(\R; \dot H^{s_c}_x\times\dot H^{s_c-1}_x)}<\infty \qtq{and} \inf_{t\in \R} N(t)\geq 1.
\end{align}
Then
\begin{align}\label{quant:decay}
\sup_{t\in\R} \int_{|x-x(t)|\geq R} |u(t,x)|^{\frac{3p}2} \,dx \lesssim_u R^{-\gamma}
\end{align}
for any $\gamma < \frac{6p^2-20p+16}{3p^2}$ and in particular, for some $\gamma>1$ when $p\geq 6$.
\end{proposition}

\begin{proof}
We prove this by bootstrap; the requisite smallness comes from compactness.  We elaborate on this, before launching into the main part of the
argument.

Let $\phi:\R^3\to[0,1]$ be a smooth function with $\phi(x)=1$ when $|x|\geq 1$ and $\phi(x)=0$ when $|x|<\frac12$.  As
$u$ is almost periodic modulo symmetries and \eqref{quant:input} holds, for any $\eta>0$ we may choose $R_0$ so that
\begin{align}\label{E:large freq small}
\sup_{t\in\R} \Bigl\{  \bigl\| \phi\bigl(\tfrac{10}{R_0} [x-x(t)]\bigr) u \bigr\|_{\dot H^{s_c}}
    + \bigl\| \phi\bigl(\tfrac{10}{R_0} [x-x(t)]\bigr) u_t \bigr\|_{\dot H^{s_c-1}} \Bigr\} \leq \eta.
\end{align}
Requiring $\eta$ to be small enough that the small-data global well-posedness theory applies and making use of simple
domain of dependence arguments, we deduce that
\begin{equation}\label{quant:small Strichartz}
\sup_{T\in\R} \Bigl\{\bigl\| u \bigr\|_{L_t^{\frac{4p(p-1)}{3p+2}} L_x^{\frac{12p(p-1)}{5(p-2)}}(\{|x-x(T)|\geq \frac{R_0}{10}+|t-T|\})} +
    \bigl\| \nabla u \bigr\|_{L^\infty_t L^{\frac{3p}{p+2}}(\{|x-x(T)|\geq \frac{R_0}{10}+|t-T|\})}  \Bigl\} \lesssim \eta.
\end{equation}

We now turn to the main part of the proof of Proposition~\ref{P:Spatial decay}. By the time-translation symmetry of the
problem, it suffices to consider a single time, say $t=0$.  By space-translation symmetry, we may set $x(0)=0$.
Using Lemma~\ref{L:duhamel} we may represent $u(0)$ as an integral over
$[0,\infty)$, which we choose to break here into two pieces: $[0,\delta R]\cup[\delta R,\infty)$.  Thus $u(0)=f+g$ with
\begin{equation}\label{quant:f+g}
f := \int_{\delta R}^\infty \frac{\sin\bigl(t\sqrtDelta\bigr)}{\sqrtDelta} F(u(t))\,dt
\qtq{and}
g := \int_0^{\delta R} \frac{\sin\bigl(t\sqrtDelta\bigr)}{\sqrtDelta} F(u(t))\,dt .
\end{equation}
Here $\delta>0$ is a small number that will be chosen in due course.

The estimate we need for the long-time piece $f$, was obtained already in \eqref{E:LongTimeEstimate}:
\begin{equation}\label{quant:long time}
\| f \|_{L^\infty_x(\R^3)} \lesssim_u (\delta R)^{-\frac2p}.
\end{equation}
By contrast, we estimate $g$ in a more natural (scale-invariant) space.  Note that by finite speed of propagation, both for the propagator
$|\nabla|^{-1}\sin(t|\nabla|)$  (cf. Lemma~\ref{L:propagator}) as well as for the center $x(t)$ of the wave packet (cf. Proposition~\ref{P:Standard}),
we see that for $|x|\geq R$ the value of $g(x)$ depends only on the values of $u$ in the set
$$
\Omega_R := \{(t,x) : t\in[0,\delta R]\text{ and } |x-x(t)|\geq (1-2\delta)R - 2C_u\},
$$
where $C_u$ is as in \eqref{x is Lip}.  With this in hand, we now estimate $g$ using Sobolev embedding (which is valid on the complement of a ball)
together with the Strichartz and H\"older inequalities:
\begin{align}\label{quant:short time A}
\| g \|_{L^{3p/2}_x(|x|\geq R)}
   &\lesssim \biggl\| \int_0^{\delta R} \frac{\sin\bigl(t\sqrtDelta\bigr)}{\sqrtDelta} \nabla F(u(t))\,dt \biggr\|_{\dot H^{s_c-1}(|x|\geq R)} \notag\\
&\lesssim \bigl\| \nabla F(u) \bigr\|_{L_{t,x}^{\frac{4p}{3p+2}}(\Omega_R)} \notag\\
&\lesssim \bigl\| u \bigr\|_{L^\infty_t L_x^{3p/2}(\Omega_R)} \bigl\| u \bigr\|^{p-1}_{L_t^{\frac{4p(p-1)}{3p+2}} L_x^{\frac{12p(p-1)}{5(p-2)}}(\Omega_R)}
     \bigl\| \nabla u \bigr\|_{L^\infty_t L^{\frac{3p}{p+2}}(\Omega_R)}
\end{align}
Now requiring $R>R_0(u)\geq 8C_u$ and $\delta\leq 1/8$, we see that $\Omega_R$ is included in the region
where $|x-x(t)|\geq R/2$, which we apply to the first copy of $u$. If $R\geq R_0$, then $\Omega_R$ also is included in the region covered by
\eqref{quant:small Strichartz}, which we apply to the next two factors.  In this way we obtain
\begin{align}\label{quant:short time}
    \| g \|_{L^{3p/2}_x(|x|\geq R)} &\lesssim \eta^{p} \bigl\| u \bigr\|_{L^\infty_t L_x^{3p/2}(\R\times\{|x-x(t)|\geq R/2\})}
\end{align}
for a fixed small $\delta$ and $R\geq R_0(\eta,u)$.

Next we put the two pieces, $f$ and $g$, together to bound LHS\eqref{quant:decay}.  This is a simple application of
standard tricks from real interpolation: Fix $A>0$ so that $\|f\|_{L^\infty_x} \leq A/2$; note that by \eqref{quant:long time}, $A\lesssim_u (\delta R)^{-2/p}$.
Then
$$
|u(0,x)|^{3p/2} \leq A^{\frac{3p-2q}2}|u(0,x)|^{q} + |2 g(x)|^{3p/2}
$$
and so, using \eqref{quant:short time},
$$
\int_{|x|\geq R} |u(0,x)|^{3p/2} \,dx \lesssim_{u,\delta}  R^{ -\frac{3p-2q}{p} } \|u(0,x)\|^{q}_{L^q_x(\R^3)}
    + \eta^{p} \sup_t \int_{|x-x(t)|\geq R/2} |u(t,x)|^{3p/2} \,dx
$$
for $R\geq R_0(\eta,u)$.

We now have our basic inductive step.  Defining
$$
B(R) := \sup_{t\in\R} \int_{|x-x(t)|\geq R} |u(t,x)|^{3p/2} \,dx,
$$
restoring space- and time-translation invariance, and invoking Proposition~\ref{P:L^p breach}, we have
$$
B(R) \lesssim_{u,\delta}  R^{-\frac{3p-2q}{p}} + \eta^{p} B(\tfrac12 R) \qtq{for any} q > \tfrac{3p^2+20p-16}{6p}
$$
and $R\geq R_0(\eta,u)$.  On the other hand, by \eqref{quant:input} and Sobolev embedding, $B(R)\lesssim_u 1$ for $R\leq R_0(\eta,u)$.
The desired estimate now follows by choosing $\eta$ sufficiently small and performing a simple induction.
\end{proof}

%%%%%%%%%%%%%%%%%%%%%%%%%%%%%%%%%%%%%%%%%%%%%%%%%%%%%%%%%%%%%%%%%%%%%%%%%%%%%%%%%%%%%%%%%%%
%
%
%                                   Section
%
%
%%%%%%%%%%%%%%%%%%%%%%%%%%%%%%%%%%%%%%%%%%%%%%%%%%%%%%%%%%%%%%%%%%%%%%%%%%%%%%%%%%%%%%%%%%%

\section{Global enemies have finite energy}\label{S:finite E}

In this section, we prove that the soliton-like and frequency-cascade solutions described in Theorem~\ref{T:enemies} have finite energy,
that is, $\ntx u$ is square integrable.  The first and main step is the following:

\begin{theorem}\label{T:reduce s}
Let $u$ be a global solution to \eqref{nlw} that is almost periodic modulo symmetries.  Assume also that
$$
\inf_{t\in\R} N(t) \geq 1
$$
and
\begin{align}\label{energy:s bound}
\bigl\||\nabla|^{s-1} \ntx u\bigr\|_{L_t^\infty L_x^2}<\infty
\end{align}
for some $1< s \leq s_c$.  Then for all $0<\eps\leq\eps_0(p)$
\begin{equation}\label{energy:goal}
\bigl\||\nabla|^{s-1-\eps} \ntx u\bigr\|_{L_t^\infty L_x^2}<\infty,
\end{equation}
provided $s-1-\eps> 0$.
\end{theorem}

\begin{proof}
By time-translation symmetry, it suffices to prove the claim for $t=0$.  By space-translation symmetry, we may also assume $x(0)=0$.

By Bernstein's inequality and \eqref{energy:s bound}, it suffices to prove the following space-localized low-frequency bound: For some
$\beta=\beta(\eps)>0$,
\begin{align}\label{energy:enough}
\bigl\|\theta(i\nabla) P_{\leq 1}  |\nabla|^{s-1-\eps} \ntx u(0)\bigr\|_{L_x^2(B_R)}^2 \lesssim_{u, \eps} R^{-3\beta}
\end{align}
uniformly for $R\geq R_0(u)$ and `Whitney' balls $B_R = \{ x\in\R^3 : |x-x_0| \leq R\}$ with $|x_0|=3R$.  Here $\theta$ is as defined in \eqref{theta} and
plays the role of a low-frequency projection, but one whose convolution kernel has compact support; the utility of this fact will be apparent in due course
and is responsible for the appearance of $\theta$ in Proposition~\ref{P:weakD}.

To see that \eqref{energy:enough} really does suffice, we note that
\begin{align}\label{theta bdd}
\bigl\| \theta(i\nabla) P_{\leq 1} f\|_{L_x^2}\lesssim 1 \quad \implies \quad \bigl\| P_{\leq 1} f\|_{L_x^2}\lesssim 1
\end{align}
because $|\theta(\xi)|\gtrsim 1$ for $|\xi|\leq \frac{11}{10}$, which is the Fourier support of $P_{\leq 1}$.

To obtain \eqref{energy:enough} we use both Duhamel formulae in \eqref{Duhamel} to write:
\begin{align}\label{monster}
& \bigl\|\theta(i\nabla) P_{\leq 1}  |\nabla|^{s-1-\eps} \ntx u(0)\bigr\|_{L_x^2(B_R)}^2\\
={} & - \int_0^\infty \int_{-\infty}^0 \bigl\langle \nabla \tfrac{\sin(|\nabla|t)}{|\nabla|}\theta(i\nabla) |\nabla|^{s-1-\eps} F_{\leq1}(t),\ %
        \chi_{R}\nabla \tfrac{\sin(|\nabla|\tau)}{|\nabla|}\theta(i\nabla) |\nabla|^{s-1-\eps} F_{\leq1}(\tau)\bigr\rangle\,d\tau\,dt \notag\\
&     - \int_0^\infty \int_{-\infty}^0 \bigl\langle \cos(|\nabla|t) \theta(i\nabla)|\nabla|^{s-1-\eps}F_{\leq1}(t),\ %
        \chi_{R}\cos(|\nabla|\tau) \theta(i\nabla)|\nabla|^{s-1-\eps} F_{\leq1}(\tau)\bigr\rangle\,d\tau\,dt, \notag
\end{align}
where $\chi_{R}$ is a smooth cutoff function associated to the ball $B_R$.  More precisely, we set
$$
\chi_{R}(x) = \phi\bigl(\tfrac{x-x_0}{R}\bigr)
$$
where $\phi:\R^3\to[0,1]$ is a smooth function obeying $\phi(x)=1$ for $|x|\leq1$ and $\phi(x)=0$ for $|x|>\tfrac{11}{10}$.
Recall that $x_0$ denotes the center of $B_R$ and obeys $|x_0|=3R$.

In order to bound the time integrals, we need to use the fact that we can bound Strichartz norms of $u$ far from $x(t)$.
We used this argument already in Section~\ref{S:quant} but will repeat the details here.  By choosing $R_0$ sufficiently large,
compactness of our solution guarantees that
\begin{align}\label{energy:large freq small}
\sup_{t\in\R} \Bigl\{  \bigl\| \phi^c\bigl(\tfrac{22}{\delta R} [x-x(t)]\bigr) u \bigr\|_{\dot H^{s_c}_x}
    + \bigl\| \phi^c\bigl(\tfrac{22}{\delta R} [x-x(t)]\bigr) u_t \bigr\|_{\dot H^{s_c-1}_x} \Bigr\} \leq \eta,
\end{align}
where $\phi^c=1-\phi$ and $\delta=\delta(u)>0$ denotes the sub-luminality constant from Proposition~\ref{P:subluminal}.
Requiring $\eta$ to be small enough that the small data global well-posedness theory applies and making use of simple
domain of dependence arguments, we deduce that
\begin{equation}\label{energy:small Strichartz}
\sup_{T\in\R} \Bigl\{ \bigl\| u \bigr\|_{L_t^\infty L_x^{3p/2}(\{|x-x(T)|\geq \frac{\delta R}{20}+|t-T|\})} +
    \bigl\| u \bigr\|_{L_t^{p/2} L_x^\infty(\{|x-x(T)|\geq \frac{\delta R}{20}+|t-T|\})}  \Bigl\} \lesssim \eta.
\end{equation}

Returning to \eqref{monster} and using the strong Huygens principle and the fact that $\supp \check{\theta}\subseteq [-4,4]^3$,
we see that we can insert a smooth cutoff $\chi_{\dod}$ to the appropriate domain of dependence in the middle of each
of the four products $\theta(i\nabla) |\nabla|^{s-1-\eps}$, specifically to the set of spacetime points that have a
light ray connecting them to a point of the form $(t,x)$ with $t=0$ and $\dist(x, \supp \chi_R)\leq 4\sqrt 3$. In
particular, we can choose $\chi_{\dod}$ so that
\begin{gather*}
\supp{\chi_{\dod}}\subseteq \{(t,x)\in \R\times\R^3:\, (1-\tfrac{\delta}{10^6})|t|-\tfrac65 R \leq |x-x_0|\leq (1+\tfrac{\delta}{10^6})|t|+\tfrac65 R\}, \\
\chi_{\dod}(t,x)=1 \qtq{when} \bigl| |x-x_0| - |t| \bigr| \leq \tfrac{6}{5}R\\
\text{and} \quad  |\partial^\alpha \chi_{\dod}|\lesssim_{\delta,\alpha} (|t|+R)^{-|\alpha|}
\end{gather*}
for all multi-indices $\alpha$.  We will also need a slightly fattened version of $\chi_{\dod}$ which we call $\tilde \chi_{\dod}$.  It is defined so that
\begin{gather*}
\tilde\chi_{\dod}(t,x)=1 \qtq{when} \dist(x,\supp\chi_{\dod}(t) )\leq \tfrac1{10} R + \tfrac{\delta}{10^6}|t|,\\
\supp{\tilde\chi_{\dod}}\subseteq \{(t,x):\, (1-\tfrac{3\delta}{10^6})|t|-\tfrac85 R \leq |x-x_0|\leq (1+\tfrac{3\delta}{10^6})|t|+\tfrac85 R\},\\
\text{and} \quad |\partial^\alpha_x \tilde\chi_{\dod}| \lesssim_{\alpha, \delta} (|t|+R)^{-|\alpha|}
\end{gather*}
for all multi-indices $\alpha$.

By \eqref{x is Lip}, we have $|x(t)|\leq |t| + 2C_u$ for all $t\in \R$ and thus, for $R$ large enough (to defeat~$C_u$),
\begin{align}\label{sep 1}
\dist(x(t),\supp \tilde\chi_{\dod}(t))\geq \tfrac{\delta}{20}(|t| + R) \qtq{for} 0\leq |t|\leq \tfrac{R}2.
\end{align}
On the other hand, by the sub-luminality bound \eqref{E:subluminal}, we know that $|x(t)|\leq (1-\delta)|t|$ for $|t|\geq 1/\delta$.  Thus,
$$
|x(t)-x_0|\leq 3R + (1-\delta)|t| \quad \text{for}\quad |t|\geq \tfrac1{\delta}
$$
and hence, for $|t|\geq \tfrac{10}{\delta}$R,
\begin{align}\label{sep 2}
\dist(x(t),\supp \tilde\chi_{\dod}(t))\geq \bigl((1-\tfrac{3\delta}{10^6})|t|-\tfrac85 R\bigr) - \bigl(3R + (1-\delta)|t|\bigr)\geq \tfrac{\delta}{20}(|t|+R).
\end{align}

The most dangerous regime is when $|t|\in [\frac R2, \frac{10}{\delta}R]$, for then $x(t)$ may lie near (indeed inside) the support of $\chi_{\dod}$.
Here we make use of a further smooth partition of unity, namely, $1=\chi_{\near}+ \chi_{\far}$ with
$$
\supp\chi_{\near}\subseteq\{ (t,x):\, \tfrac R2\leq |t| \leq \tfrac{10}{\delta}R \text{ and } |x-x(t)|\leq \tfrac\delta{5} (|t|+R)\},
$$
and
$$
\chi_{\far}(t,x)=0 \qtq{when} \tfrac R2\leq |t| \leq \tfrac{10}{\delta}R \qtq{and} |x-x(t)|\leq \tfrac\delta{10} (|t|+R).
$$
Note that this can be done in a manner such that
$$
|\partial^\alpha_x \chi_{\near}| + |\partial^\alpha_x \chi_{\far}|\lesssim_{\alpha, \delta} (|t|+R)^{-|\alpha|}.
$$
for all multi-indices $\alpha$.  It is for the sake of notational convenience that we have defined $\chi_{\near}$ to be identically equal to zero for
$t$ outside the region $\tfrac R2\leq |t| \leq \tfrac{10}{\delta}R$ and correspondingly, $\chi_{\far}\equiv 1$ there.

We will also need a slightly fattened version $\tilde \chi_\far$ of $\chi_{\far}$, chosen so that
\begin{gather*}
\tilde\chi_{\far}(t,x)=1 \qtq{when} \dist(x,\supp\chi_{\far}(t))\leq \tfrac{\delta}{40}(|t|+R) ,\\
\supp \tilde\chi_{\far}(t) \subseteq \bigl\{ x:\, |x-x(t)|\geq \tfrac\delta{20} (|t|+R)\bigr\}   \qtq{when}  \tfrac12 R \leq |t| \leq \tfrac{10}{\delta}R ,\\
\text{and} \quad |\partial^\alpha_x \tilde\chi_{\far}|\lesssim_{\alpha, \delta} (|t|+R)^{-|\alpha|},
\end{gather*}
for all multi-indices $\alpha$.

Collecting \eqref{sep 1}, \eqref{sep 2}, and the definition of $\tilde \chi_\far$, we note that
\begin{align}\label{sep final}
\dist(x(t),\supp [\tilde\chi_{\dod}\tilde \chi_\far](t))\geq \tfrac{\delta}{20}(|t|+R).
\end{align}
On the other hand, for $(t,x), (\tau,y)\in \supp \chi_{\near}$ with $t>0$ and $\tau<0$, we obtain
\begin{align}\label{big O R}
|t|+|\tau|+|x|+|y| \lesssim_{\delta} R
\end{align}
and, more importantly,
\begin{align}\label{separation}
|t-\tau|-|x-y|
&\geq |t-\tau|- (1-\delta)|t-\tau| - \tfrac{\delta}{5}(t+R) - \tfrac{\delta}{5}(|\tau| + R)\notag\\
&\geq \tfrac{4\delta}{5}|t-\tau| -\tfrac{2\delta}{5}R\geq \tfrac{2\delta}{5}R,
\end{align}
by sub-luminality and taking $R>R_0(u)\geq1/\delta$. The significance of these inequalities is that they allow us to apply Proposition~\ref{P:weakD};
see Lemma~\ref{L:near-near} below.

Before beginning to estimate \eqref{monster}, we first note some consequences of \eqref{energy:small Strichartz} in terms of our cutoffs:

\begin{lemma}\label{L:Strichartz far}
Under the assumptions above (and taking $R_0$ even larger if necessary),
\begin{align*}
\bigl\| \tilde\chi_{\dod}\tilde\chi_{\far} u \bigr\|_{L_t^\infty L_x^{3p/2}(I\times\R^3)} +
    \bigl\| \tilde\chi_{\dod}\tilde\chi_{\far} u \bigr\|_{L_t^{p/2} L_x^\infty(I\times\R^3)} \lesssim_u 1,
\end{align*}
uniformly for $I=[-\tfrac{10}{\delta}R,\tfrac{10}{\delta}R]$ or $I=[T,2T] \cup [-2T,-T]$ with $T\geq \tfrac{10}{\delta}R$.
\end{lemma}

\begin{proof}
We will only prove the claim for positive times $t$.  For negative times, the argument is similar.

Recall that by \eqref{energy:small Strichartz},
\begin{equation}\label{energy:sS b}
\sup_{T\in\R} \Bigl\{ \bigl\| u \bigr\|_{L_t^\infty L_x^{3p/2}(\{|x-x(T)|\geq \frac{\delta R}{20}+|t-T|\})} +
    \bigl\| u \bigr\|_{L_t^{p/2} L_x^\infty(\{|x-x(T)|\geq \frac{\delta R}{20}+|t-T|\})}  \Bigl\} \lesssim \eta.
\end{equation}
Thus, choosing $T=0$, we obtain
\begin{align*}
\bigl\| \tilde\chi_{\dod}\tilde\chi_{\far} u \bigr\|_{L_t^\infty L_x^{3p/2}([0,\frac R2]\times\R^3)} +
    \bigl\| \tilde\chi_{\dod}\tilde\chi_{\far} u \bigr\|_{L_t^{p/2} L_x^\infty([0,\frac R2]\times\R^3)} \lesssim \eta
\end{align*}
since if $0\leq t \leq \tfrac R2$ and $\tilde\chi_\dod(t,x)\neq 0$, then  $|x|\geq \frac45 R$.

On the other hand, choosing $T\geq \frac{10}{\delta}R$ gives
\begin{align*}
\bigl\| \tilde\chi_{\dod}\tilde\chi_{\far} u \bigr\|_{L_t^\infty L_x^{3p/2}([T,2T]\times\R^3)} +
    \bigl\| \tilde\chi_{\dod}\tilde\chi_{\far} u \bigr\|_{L_t^{p/2} L_x^\infty([T,2T]\times\R^3)} \lesssim \eta.
\end{align*}
Indeed, using \eqref{E:subluminal}, for $(t,x)\in \supp \tilde\chi_{\dod}\tilde\chi_{\far}$ and $t\geq T \geq \frac{10}{\delta}R$, we have
\begin{align*}
|x-x(T)| &\geq |x-x_0| - |x_0| - |x(T)| \\
& \geq (1-\tfrac{3\delta}{10^6})|t-T| - (3+\tfrac85) R + (1-\tfrac{3}{10^6})\delta T  \\
&\geq  (1-\tfrac{3\delta}{10^6})|t-T| + \tfrac{\delta R}{20} + \tfrac{\delta}{50}T
\end{align*}
and so $|x-x(T)| \geq |t-T| +\tfrac{\delta R}{20}$ provided $\tfrac{3\delta}{10^6}|t-T|\leq \tfrac{\delta}{50} T$, which is true when $t\in[T,2T]$.

It remains to consider $\tfrac R2\leq t\leq \tfrac{10}{\delta}R$.  For this region, we choose a mesh
$$
\tfrac R2=T_0<T_1<\cdots<T_K=\tfrac{10}{\delta}R \qtq{with} \tfrac{\delta}{200}R \leq |T_k-T_{k-1}|\leq \tfrac{\delta}{100}R \qtq{for all} 1\leq k\leq K.
$$
Note that $K\lesssim \delta^{-2}$.  Then for $(t,x)\in \supp \tilde\chi_{\dod}\tilde\chi_{\far}$ with $t\in [T_{k-1},T_k]$,
$$
|x-x(T_{k-1})|\geq |x-x(t)|-|x(t)-x(T_{k-1})|\geq \tfrac{\delta}{20}(|t|+R) - |t-T_{k-1}| -2C_u \geq \tfrac{\delta}{20} R + |t-T_{k-1}|,
$$
provided $R\geq R_0$ with $R_0$ sufficiently large depending on $u$.  Thus, using \eqref{energy:sS b} with $T=T_k$ for $0\leq k\leq K-1$
and summing, we derive
$$
\bigl\| \tilde\chi_{\dod}\tilde\chi_{\far} u \bigr\|_{L_t^\infty L_x^{3p/2}([\frac R2,\frac{10}{\delta}R]\times\R^3)} +
    \bigl\| \tilde\chi_{\dod}\tilde\chi_{\far} u \bigr\|_{L_t^{p/2} L_x^\infty([\frac R2,\frac{10}{\delta}R]\times\R^3)}
    \lesssim \eta \delta^{-2}  \lesssim_u  1.
$$

This concludes the proof of the lemma.
\end{proof}

After breaking up the integrals in \eqref{monster} by introducing the cutoffs $\chi_\dod\chi_\near$ and $\chi_\dod\chi_\far$, the required estimate follows directly
from the next three lemmas.

\begin{lemma}\label{L:far}
Under the assumptions above, for some small $\beta=\beta(\eps)>0$ we have
\begin{align*}
\Bigl\| \int_0^\infty \nabla \tfrac{\sin(|\nabla|t)}{|\nabla|}&\theta(i\nabla)\chi_{\dod}\chi_{\far} |\nabla|^{s-1-\eps} F_{\leq1}(t)\, dt \Bigr\|_{L_x^2(\R^3)}\\
 &+ \Bigl\| \int_0^\infty \cos(|\nabla|t)\theta(i\nabla)\chi_{\dod}\chi_{\far} |\nabla|^{s-1-\eps} F_{\leq1}(t) \, dt\Bigr\|_{L_x^2(\R^3)}
\lesssim_u R^{-1/2-4\beta}
\end{align*}
and similarly,
\begin{align*}
\Bigl\| \int_{-\infty}^0 \nabla \tfrac{\sin(|\nabla|\tau)}{|\nabla|}&\theta(i\nabla)\chi_{\dod}\chi_{\far} |\nabla|^{s-1-\eps} F_{\leq1}(\tau) \, d\tau
\Bigr\|_{L_x^2(\R^3)}\\
 &+ \Bigl\| \int_{-\infty}^0 \cos(|\nabla|\tau)\theta(i\nabla)\chi_{\dod}\chi_{\far} |\nabla|^{s-1-\eps} F_{\leq1}(\tau) \, d\tau\Bigr\|_{L_x^2(\R^3)}
\lesssim_u R^{-1/2-4\beta}.
\end{align*}
\end{lemma}

\begin{proof}
We will only present the proof of the first inequality; for negative times, the argument is similar.

For the remainder of this proof, all spacetime norms are over the region $[0,\infty)\times \R^3$.  Also, to ease
notation we write $\chi=\sqrt{\chi_{\dod}\chi_{\far}}$ and $\tilde\chi=\tilde \chi_{\dod}\tilde \chi_{\far}$.

Using the $L^2_x$-boundedness of $\theta(i\nabla)$, then the Strichartz inequality (with $6\leq q<\infty$) followed by H\"older's inequality, we obtain
\begin{align}\label{puach}
\Bigl\| \int_0^\infty & \nabla \tfrac{\sin(|\nabla|t)}{|\nabla|}\theta(i\nabla)\chi^2  |\nabla|^{s-1-\eps} F_{\leq1}(t) \, dt\Bigr\|_2
+ \Bigl\| \int_0^\infty \cos(|\nabla|t)\theta(i\nabla)\chi^2 |\nabla|^{s-1-\eps} F_{\leq1}(t) \, dt\Bigr\|_2 \notag\\
&\lesssim \bigl\| \nabla \bigl[\chi^2 |\nabla|^{s-1-\eps} F_{\leq1}\bigr] \bigr\|_{L_t^{\frac{2q}{q+6}}L_x^{\frac{q}{q-1}}}\notag\\
&\lesssim \bigl\| \chi^2\nabla|\nabla|^{s-1-\eps} F_{\leq1} \bigr\|_{L_t^{\frac{2q}{q+6}}L_x^{\frac{q}{q-1}}}
+\| \nabla \chi\|_{L_x^3} \bigl\|\chi|\nabla|^{s-1-\eps}  F_{\leq1} \bigr\|_{L_t^{\frac{2q}{q+6}}L_x^{\frac{3q}{2q-3}}}\notag\\
&\lesssim \bigl\| \chi\nabla|\nabla|^{s-1-\eps} F_{\leq1} \bigr\|_{L_t^{\frac{2q}{q+6}}L_x^{\frac{q}{q-1}}}
+\bigl\|\chi|\nabla|^{s-1-\eps}  F_{\leq1} \bigr\|_{L_t^{\frac{2q}{q+6}}L_x^{\frac{3q}{2q-3}}}\notag\\
&\lesssim \bigl\| \nabla|\nabla|^{s-1-\eps} F(\tilde \chi u)\bigr\|_{L_t^{\frac{2q}{q+6}}L_x^{\frac{q}{q-1}}}
+\bigl\||\nabla|^{s-1-\eps}  F(\tilde \chi u)\bigr\|_{L_t^{\frac{2q}{q+6}}L_x^{\frac{3q}{2q-3}}}\notag\\
&\quad +\bigl\| \chi\nabla|\nabla|^{s-1-\eps}  P_{\leq1}  (\tilde \chi^c F) \bigr\|_{L_t^{\frac{2q}{q+6}}L_x^{\frac{q}{q-1}}}
+\bigl\|\chi|\nabla|^{s-1-\eps}  P_{\leq1}  (\tilde \chi^c F) \bigr\|_{L_t^{\frac{2q}{q+6}}L_x^{\frac{3q}{2q-3}}},
\end{align}
where $\tilde\chi^c=1-\tilde \chi^{p+1}$ so that $F(u)=F(\tilde \chi u) + \tilde \chi^cF(u)$.

To complete the proof, we have to show that each of the four terms appearing on the right-hand side of \eqref{puach} is bounded by $R^{-1/2-4\beta}$.
In order to appeal to Lemma~\ref{L:Strichartz far}, we will sometimes need to partition $[0,\infty)$ into the collection of intervals $I_j=[T_j,T_{j+1}]$
with $T_0:=0$ and $T_j=\tfrac{10}{\delta}R 2^{j-1}$ for all $j\geq 1$.

We start with the first term in RHS\eqref{puach}.  By the fractional chain and product rules together with H\"older's inequality, Sobolev embedding,
Lemma~\ref{L:Strichartz far}, Proposition~\ref{P:L^p breach}, and the combination of Proposition~\ref{P:Spatial decay} and~\eqref{sep final},
\begin{align*}
&\bigl\| \nabla|\nabla|^{s-1-\eps} F(\tilde \chi u)\bigr\|_{L_t^{\frac{2q}{q+6}}L_x^{\frac{q}{q-1}}}\\
&\lesssim \sum_{j\geq 0}\bigl\| \nabla|\nabla|^{s-1-\eps} F(\tilde \chi u)\bigr\|_{L_t^{\frac{2q}{q+6}}L_x^{\frac{q}{q-1}}(I_j\times\R^3)}\\
&\lesssim \sum_{j\geq 0} \Bigl[\bigl\| |\nabla|^{s-\eps}u \bigr\|_{L_t^\infty L_x^{\frac{2q}{q-1}}(I_j\times\R^3)}
+ \bigl\| |\nabla|^{s-\eps}\tilde \chi\bigr\|_{L_t^\infty L_x^{\frac{3}{s-\eps}}(I_j\times\R^3)}
    \|u\|_{L_t^\infty L_x^{\frac{6q}{q(3-2s+2\eps)-3}}(I_j\times\R^3)}\Bigr]\\
&\qquad \qquad \qquad \times \|\tilde \chi u\|_{L_t^{\frac p2}L_x^\infty(I_j\times\R^3)}^{\frac{p(q+6)}{4q}}
    \|\tilde \chi u\|_{L_t^\infty L_x^p(I_j\times\R^3)}^{\frac{3p}{2q}}\|\tilde \chi u\|_{L_t^\infty L_x^{\frac{3p}2}(I_j\times\R^3)}^{\frac{3p(q-4)}{4q}}\\
&\lesssim_u \sum_{j\geq 0} (T_j+R)^{-\frac{(q-4)\gamma}{2q}}\bigl\| |\nabla|^{s-\eps}u \bigr\|_{L_t^\infty L_x^{\frac{2q}{q-1}}(I_j\times\R^3)}\\
&\lesssim_u R^{-\frac{(q-4)\gamma}{2q}}.
\end{align*}
In the last step, we used Sobolev embedding followed by interpolation, \eqref{energy:s bound}, and Proposition~\ref{P:L^p breach}.  This requires
$\tfrac{3ps}{(3p-4)\eps}\leq  q<\tfrac{(3p^2+20p-16)s}{(3p^2+8p-16)\eps}$.  In order to make the power of $R$ less than $-1/2$, it suffices to take $q$ large,
which in turn forces $\eps$ to be small.

Arguing similarly, we estimate the second term in RHS\eqref{puach} as follows:
\begin{align*}
&\bigl\||\nabla|^{s-1-\eps}  F(\tilde \chi u)\bigr\|_{L_t^{\frac{2q}{q+6}}L_x^{\frac{3q}{2q-3}}}\\
&\lesssim \sum_{j\geq 0} \bigl\||\nabla|^{s-1-\eps}  F(\tilde \chi u)\bigr\|_{L_t^{\frac{2q}{q+6}}L_x^{\frac{3q}{2q-3}}(I_j\times\R^3)}\\
&\lesssim \sum_{j\geq 0} \Bigl[\bigl\| |\nabla|^{s-1-\eps}u \bigr\|_{L_t^\infty L_x^{\frac{6q}{q-3}}(I_j\times\R^3)}
+ \bigl\| |\nabla|^{s-1-\eps}\tilde \chi\bigr\|_{L_t^\infty L_x^{\frac{3}{s-1-\eps}}(I_j\times\R^3)}
    \|u\|_{L_t^\infty L_x^{\frac{6q}{q(3-2s+2\eps)-3}}(I_j\times\R^3)}\Bigr]\\
&\qquad \qquad \qquad \times \|\tilde \chi u\|_{L_t^{\frac p2}L_x^\infty(I_j\times\R^3)}^{\frac{p(q+6)}{4q}}
    \|\tilde \chi u\|_{L_t^\infty L_x^p(I_j\times\R^3)}^{\frac{3p}{2q}}\|\tilde \chi u\|_{L_t^\infty L_x^{\frac{3p}2}(I_j\times\R^3)}^{\frac{3p(q-4)}{4q}}\\
&\lesssim_u \sum_{j\geq 0} (T_j+R)^{-\frac{(q-4)\gamma}{2q}}\bigl\| |\nabla|^{s-\eps}u \bigr\|_{L_t^\infty L_x^{\frac{2q}{q-1}}(I_j\times\R^3)}\\
&\lesssim_u R^{-\frac{(q-4)\gamma}{2q}},
\end{align*}
which again yields the desired decay in $R$ for $q$ large enough.

In order to estimate the remaining two terms in RHS\eqref{puach}, we note that
\begin{gather*}
\dist(\supp \chi, \supp\tilde \chi^c) \gtrsim_\delta |t|+ R.
\end{gather*}
Hence, by the mismatch estimate Lemma~\ref{L:mismatch} together with Proposition~\ref{P:L^p breach},
\begin{align*}
\bigl\|\chi \nabla |\nabla|^{s-1-\eps} P_{\leq1} (\tilde \chi^c F) \bigr\|_{L_t^{\frac{2q}{q+6}}L_x^{\frac{q}{q-1}}}
&+\bigl\|\chi |\nabla|^{s-1-\eps} P_{\leq1} (\tilde \chi^c F) \bigr\|_{L_t^{\frac{2q}{q+6}}L_x^{\frac{3q}{2q-3}}}\\
&\lesssim \|F(u)\|_{L_t^\infty L_x^1} \bigl\| (|t|+R)^{-(s-\eps)-\frac3q}\bigr\|_{L_t^{\frac{2q}{q+6}}([0,\infty))}\\
&\lesssim_u R^{-\frac12-4\beta},
\end{align*}
provided $4\beta<s-1-\eps$.  This finishes the proof of the lemma.
\end{proof}

\begin{lemma}\label{L:near}
Under the assumptions above, for any $\beta>0$ we have
\begin{align*}
\Bigl\| \int_0^\infty \nabla\tfrac{\sin(|\nabla|t)}{|\nabla|}&\theta(i\nabla)\chi_{\dod}\chi_{\near} |\nabla|^{s-1-\eps} F_{\leq1}(t) \, dt\Bigr\|_{L_x^2(\R^3)}\\
 &+ \Bigl\| \int_0^\infty \cos(|\nabla|t)\theta(i\nabla)\chi_{\dod}\chi_{\near} |\nabla|^{s-1-\eps} F_{\leq1}(t) \, dt\Bigr\|_{L_x^2(\R^3)}
\lesssim_u R^{1/2+ \beta}
\end{align*}
and similarly,
\begin{align*}
\Bigl\| \int_{-\infty}^0 \nabla\tfrac{\sin(|\nabla|\tau)}{|\nabla|}&\theta(i\nabla)\chi_{\dod}\chi_{\near} |\nabla|^{s-1-\eps} F_{\leq1}(\tau) \, d\tau\Bigr
\|_{L_x^2(\R^3)}\\
 &+ \Bigl\| \int_{-\infty}^0 \cos(|\nabla|\tau)\theta(i\nabla)\chi_{\dod}\chi_{\near} |\nabla|^{s-1-\eps} F_{\leq1}(\tau)\, d\tau \Bigr\|_{L_x^2(\R^3)}
\lesssim_u R^{1/2+\beta}.
\end{align*}
\end{lemma}

\begin{proof}
Again, we present the proof for positive times only.

First, recall that $\chi_{\near}$ is supported in the spacetime region where $\frac R2\leq t\leq \frac{10}{\delta}R$.
Thus, for the remainder of this proof, all spacetime norms will be over the region $[\frac R2, \frac{10}{\delta}R]\times\R^3$.

Now, by the Strichartz inequality followed by H\"older's inequality, Sobolev embedding, Bernstein's inequality,
and Proposition~\ref{P:L^p breach}, we obtain
\begin{align*}
\Bigl\| \int_0^\infty  \nabla& \tfrac{\sin(|\nabla|t)}{|\nabla|} \theta(i\nabla)\chi_{\dod}\chi_{\near} |\nabla|^{s-1-\eps} F_{\leq1}(t) \, dt\Bigr\|_{L_x^2(\R^3)}\\
 &\qquad\qquad\qquad + \Bigl\| \int_0^\infty \cos(|\nabla|t)\theta(i\nabla)\chi_{\dod}\chi_{\near} |\nabla|^{s-1-\eps} F_{\leq1}(t)\, dt \Bigr\|_{L_x^2(\R^3)}\\
 &\lesssim \bigl\| \nabla |\nabla|^{s-1-\eps} F_{\leq1} \bigr\|_{L_t^{\frac{2q}{q+6}} L_x^{\frac q{q-1}}}
 + \bigl\|\nabla (\chi_{\dod} \chi_{\far})\bigr\|_{L_x^3} \bigl\| |\nabla|^{s-1-\eps} F_{\leq1}\bigr\|_{L_t^{\frac{2q}{q+6}} L_x^{\frac{3q}{2q-3}}}\\
 &\lesssim \bigl\| |\nabla|^{s-\eps} F_{\leq1} \bigr\|_{L_t^{\frac{2q}{q+6}} L_x^{\frac q{q-1}}}\\
 &\lesssim \bigl\|F(u) \bigr\|_{L_t^{\frac{2q}{q+6}} L_x^1}\\
 &\lesssim_u R^{\frac12+\frac3q},
\end{align*}
for any $6\leq q<\infty$.  The claim now follows by taking $q$ sufficiently large depending on $\beta$.
\end{proof}

We now turn to the most significant region of integration, where $(t,x(t))$ and $(\tau, x(\tau))$ may lie in the domain of dependence of $B_R$.

\begin{lemma}\label{L:near-near}
With the assumptions above and $\beta<1/30$, we have
\begin{align}\label{near-near}
&\Bigl|\int_0^\infty \int_{-\infty}^0 \bigl\langle \nabla\tfrac{\sin(|\nabla|t)}{|\nabla|}\theta(i\nabla) \chi_{\dod}\chi_{\near}|\nabla|^{s-1-\eps}
F_{\leq1}(t),\notag\\
&\qquad \qquad \qquad \qquad \qquad \qquad
        \chi_{R}\nabla \tfrac{\sin(|\nabla|\tau)}{|\nabla|}\theta(i\nabla) \chi_{\dod}\chi_{\near} |\nabla|^{s-1-\eps} F_{\leq1}(\tau)\bigr\rangle\,d\tau
\,dt
\\
&+\int_0^\infty \int_{-\infty}^0 \bigl\langle \cos(|\nabla|t)\theta(i\nabla) \chi_{\dod}\chi_{\near}|\nabla|^{s-1-\eps} F_{\leq1}(t),\notag\\
&\qquad \qquad \qquad \qquad \qquad \qquad
        \chi_{R} \cos(|\nabla|\tau)\theta(i\nabla) \chi_{\dod}\chi_{\near} |\nabla|^{s-1-\eps} F_{\leq1}(\tau)\bigr\rangle\,d\tau\,dt\Bigr|
\lesssim_u R^{-3\beta}.\notag
\end{align}
\end{lemma}

\begin{proof}
The claim will follow from Proposition~\ref{P:weakD}; the hypothesis \eqref{hyp} holds by virtue of \eqref{big O R} and \eqref{separation}.
Thus, using this proposition followed by Bernstein's inequality and Proposition~\ref{P:L^p breach}, we obtain
\begin{align*}
\text{LHS}\eqref{near-near}
&\lesssim R^{-1/10} \bigl\||\nabla|^{s-1-\eps} F_{\leq1}\bigr\|_{L_t^\infty L_x^1}^2
\lesssim R^{-1/10} \| F\|_{L_t^\infty L_x^1}^2
\lesssim_u R^{-1/10}.
\end{align*}
This completes the proof of the lemma.
\end{proof}

We now return to the proof of Theorem~\ref{T:reduce s}.  Recall that it suffices to prove \eqref{energy:enough}.  This follows for
$\beta < \min(\tfrac1{30},\tfrac{s-1-\eps}{4})$ by using Lemmas~\ref{L:far}, \ref{L:near}, and \ref{L:near-near} to estimate \eqref{monster}.
\end{proof}

\begin{corollary}\label{C:finite energy}
Let $u$ be a global solution to \eqref{nlw} that is almost periodic modulo symmetries.  Recall that
\begin{align}\label{crit bdd}
\bigl\|(u,u_t)\bigr\|_{L_t^\infty(\R;\dot H^{s_c}_x\times\dot H^{s_c-1}_x)}<\infty.
\end{align}
Assume also that
$$
\inf_{t\in\R} N(t) \geq 1.
$$
Then $\ntx u \in L_t^\infty L_x^2;$ in particular, the energy $E(u)$
of the solution is finite.  Moreover, there exists $\beta=\beta(p)>0$ so that
\begin{equation}\label{ultimate}
 \bigl\|\langle x-x(t)\rangle^\beta  P_{\leq 1}\ntx u \bigr\|_{L_t^\infty L_x^2} \lesssim_u 1.
\end{equation}
\end{corollary}

\begin{proof}
Applying Theorem~\ref{T:reduce s} iteratively, finitely many times, we conclude that
\begin{align}\label{s>1}
\ntx u \in L_t^\infty \dot H^{s-1}_x \qtq{for each} 1<s\leq s_c.
\end{align}
To pass from this to finite energy, we follow the strategy used in Theorem~\ref{T:reduce s}, indeed with some
simplifications due to the local nature of the operator $\nabla$ as opposed to $|\nabla|^{s-1-\eps}$.  As $P_{\leq 1}$
is also non-local, we replace it by $\theta(i\nabla)$, which is almost local.

Note that it suffices to prove \eqref{ultimate}.  Indeed, using this to bound the low-frequency part of the solution
and using \eqref{crit bdd} and Bernstein's inequality to bound the high frequencies, we deduce that
$\ntx u \in L_t^\infty L_x^2$.  This renders the first two terms in the energy \eqref{energy} finite.  Using
Sobolev embedding and interpolation between $u \in L_t^\infty \dot H^1_x$ and $u\in L_t^\infty \dot H^{s_c}_x$,
we also see that the potential energy term is finite.  Thus, $E(u)<\infty$.

Therefore, it remains to establish \eqref{ultimate}.  By time-translation symmetry, it suffices to prove the claim for $t=0$.
By space-translation symmetry, we may also assume $x(0)=0$.  Arguing as we did for \eqref{energy:enough}, it suffices to show
\begin{align}\label{energy:enough 2}
\bigl\| \theta(i\nabla)^2 \ntx u(0)\bigr\|_{L_x^2(B_R)}^2 \lesssim_u R^{-3\beta}
\end{align}
uniformly for $R\geq R_0(u)$.

To obtain \eqref{energy:enough 2} we use the Duhamel formulae \eqref{Duhamel} to write:
\begin{align}\label{monster 2}
\bigl\| \theta(i\nabla)^2 & \ntx u(0)\bigr\|_{L_x^2(B_R)}^2\\
={} & - \int_0^\infty \int_{-\infty}^0 \bigl\langle \nabla \tfrac{\sin(|\nabla|t)}{|\nabla|}\theta(i\nabla)^2 F(t),\ %
        \chi_{R} \nabla \tfrac{\sin(|\nabla|\tau)}{|\nabla|}\theta(i\nabla)^2 F(\tau)\bigr\rangle\,d\tau\,dt \notag\\
&     - \int_0^\infty \int_{-\infty}^0 \bigl\langle \cos(|\nabla|t) \theta(i\nabla)^2 F(t),\ %
        \chi_{R}\cos(|\nabla|\tau) \theta(i\nabla)^2 F(\tau)\bigr\rangle\,d\tau\,dt, \notag
\end{align}
where $\chi_{R}$ is a smooth cutoff function associated to the ball $B_R$, as previously.

To estimate \eqref{monster 2}, we decompose spacetime in exactly the same manner as in the proof of
Theorem~\ref{T:reduce s}, by introducing $\chi_{\dod}$, $\chi_\near$, and $\chi_\far$ between the two copies of
$\theta(i\nabla)$.  Lemmas~\ref{L:Strichartz far}, \ref{L:near}, and \ref{L:near-near} continue to hold when
$s-1-\eps=0$ and with the replacement of $P_{\leq 1}$ by $\theta(i\nabla)$.  In connection with this, we should note
that Bernstein inequalities continue to hold:
$$
\| \theta(i\nabla) f \|_{L^q_x(\R^3)} + \| \nabla \theta(i\nabla) f \|_{L^q_x(\R^3)} \lesssim \| f \|_{L^p_x(\R^3)}
$$
for all $1\leq p \leq q\leq \infty$.

The only part of the proof of Theorem~\ref{T:reduce s} that needs to change is the proof of Lemma~\ref{L:far}.
Corollary~\ref{C:finite energy} thus follows  from the following substitute:

\begin{lemma}
Under the hypotheses of Corollary~\ref{C:finite energy},
\begin{align}\label{far piece}
&\Bigl\| \int_0^\infty \!\!\!\nabla \tfrac{\sin(|\nabla|t)}{|\nabla|}\theta(i\nabla)\chi_{\dod}\chi_{\far} \theta(i\nabla) F(t)\,dt \Bigr\|_2
 \!\! + \Bigl\| \int_0^\infty \!\!\!\cos(|\nabla|t)\theta(i\nabla)\chi_\dod \chi_\far  \theta(i\nabla) F(t)\, dt \Bigr\|_2 \notag\\
&\Bigl\| \int_{-\infty}^0 \!\!\! \!\nabla \tfrac{\sin(|\nabla|\tau)}{|\nabla|}\theta(i\nabla) \chi_\dod \chi_\far \theta(i\nabla) F(\tau) d\tau\Bigr\|_2
 \!\! + \Bigl\| \int_{-\infty}^0 \!\!\! \! \cos(|\nabla|\tau)\theta(i\nabla)\chi_\dod \chi_\far \theta(i\nabla) F(\tau) d\tau\Bigr\|_2 \notag\\
&\qquad \qquad\qquad \qquad\lesssim_u R^{-1/2-4\beta},
\end{align}
for some $\beta=\beta(p)>0$.
\end{lemma}

\begin{proof}
We argue as in Lemma~\ref{L:far}. Again, we only present the proof for positive times. To ease notation, we write $\chi=\sqrt{\chi_{\dod}\chi_{\far}}$ and
$\tilde \chi=\tilde \chi_{\dod}\tilde \chi_{\far}$.

Using the Strichartz inequality (with $6\leq q<\infty$) followed by H\"older's inequality, we obtain
\begin{align}\label{fatty}
\Bigl\| \int_0^\infty & \nabla \tfrac{\sin(|\nabla|t)}{|\nabla|}\theta(i\nabla)\chi^2 \theta(i\nabla) F(t) \,dt\Bigr\|_{L_x^2(\R^3)}
+ \Bigl\| \int_0^\infty \cos(|\nabla|t)\theta(i\nabla)\chi^2 \theta(i\nabla) F(t) \,dt\Bigr\|_{L_x^2(\R^3)} \notag\\
&\lesssim \bigl\| \nabla \bigl[\chi^2 \theta(i\nabla) F\bigr] \bigr\|_{L_t^{\frac{2q}{q+6}}L_x^{\frac{q}{q-1}}}\notag\\
&\lesssim \bigl\|\chi^2\theta(i\nabla) \nabla F \bigr\|_{L_t^{\frac{2q}{q+6}}L_x^{\frac{q}{q-1}}}
+\| \nabla \chi\|_{L_x^3} \bigl\|\chi \, \theta(i\nabla) F \bigr\|_{L_t^{\frac{2q}{q+6}}L_x^{\frac{3q}{2q-3}}}\notag\\
&\lesssim \bigl\|\chi\theta(i\nabla) \nabla F \bigr\|_{L_t^{\frac{2q}{q+6}}L_x^{\frac{q}{q-1}}}
+\bigl\|\chi \,\theta(i\nabla) F \bigr\|_{L_t^{\frac{2q}{q+6}}L_x^{\frac{3q}{2q-3}}}.
\end{align}

Noting that the convolution kernel associated to $\theta(i\nabla)$ has compact support, we have
\begin{align*}
\bigl\|\chi \theta(i\nabla) \nabla F(t) \bigr\|_{L_x^{\frac{q}{q-1}}}
    & \lesssim \bigl\|\tilde\chi^p \nabla F(t) \bigr\|_{L_x^{\frac{q}{q-1}}}
\qtq{and}
\bigl\|\chi \theta(i\nabla) F(t) \bigr\|_{L_x^{\frac{3q}{2q-3}}}
    \lesssim \bigl\|\tilde \chi^p F(t) \bigr\|_{L_x^{\frac{3q}{2q-3}}}.
\end{align*}
Thus by Sobolev embedding, Lemma~\ref{L:Strichartz far}, Proposition~\ref{P:L^p breach}, Proposition~\ref{P:Spatial decay}
combined with \eqref{sep final}, and~\eqref{s>1},
\begin{align*}
&\text{LHS\eqref{fatty}} \\
&\lesssim \sum_{j\geq 0} \Bigl[\bigl\| \nabla u \bigr\|_{L_t^\infty L_x^{\frac{2q}{q-1}}} + \bigl\|u \bigr\|_{L_t^\infty L_x^{\frac{6q}{q-3}}} \Bigr]
    \|\tilde \chi u\|_{L_t^{\frac p2}L_x^\infty(I_j\times\R^3)}^{\frac{p(q+6)}{4q}}
    \|\tilde \chi u\|_{L_t^\infty L_x^p}^{\frac{3p}{2q}}\|\tilde \chi u\|_{L_t^\infty L_x^{\frac{3p}2}(I_j\times\R^3)}^{\frac{3p(q-4)}{4q}}\\
&\lesssim_u \bigl\| |\nabla|^{1+\frac3{2q}}u\bigr\|_{L_t^\infty L_x^2} \sum_j (T_j+R)^{-\frac{(q-4)\gamma}{2q}}\\
&\lesssim_u R^{-\frac12-4\beta},
\end{align*}
where $\gamma$ is as in Proposition~\ref{P:Spatial decay}.  To obtain the stated decay in $R$ in the last inequality,
it suffices to take $q$ sufficiently large (depending on $p$).  This proves \eqref{far piece}.
\end{proof}

This finishes the proof of the corollary.
\end{proof}

%%%%%%%%%%%%%%%%%%%%%%%%%%%%%%%%%%%%%%%%%%%%%%%%%%%%%%%%%%%%%%%%%%%%%%%%%%%%%%%%%%%%%%%%%%%
%
%
%                                   Section
%
%
%%%%%%%%%%%%%%%%%%%%%%%%%%%%%%%%%%%%%%%%%%%%%%%%%%%%%%%%%%%%%%%%%%%%%%%%%%%%%%%%%%%%%%%%%%%

\section{The frequency-cascade solution}\label{S:no cascade}

In this section, we preclude the frequency-cascade solution described in Theorem~\ref{T:enemies}.

\begin{theorem}[Absence of frequency-cascade solutions]
There are no frequency-cascade solutions to \eqref{nlw} in the sense of Theorem~\ref{T:enemies}.
\end{theorem}

\begin{proof}
We argue by contradiction.  Assume there exists a solution $u:\R\times\R^3\to \R$ that is a frequency-cascade
in the sense of Theorem~\ref{T:enemies}.  We will prove this scenario is inconsistent with the conservation of energy.

Indeed, by Corollary~\ref{C:finite energy}, the energy $E(u)$ is finite.  Next, let $0<M, \eta<1$ be small
constants to be chosen later.  By almost periodicity modulo symmetries, there exists $c(\eta)$ sufficiently small so that
\begin{align}\label{small freq comp}
\|u_{\leq c(\eta) N(t)}\|_{L_t^\infty \dot H_x^{s_c}} + \|P_{\leq c(\eta) N(t)}u_t\|_{L_t^\infty \dot H_x^{s_c-1}}\leq \eta.
\end{align}

Now decompose $u = u_{\leq M} + u_{M\leq \cdot \leq c(\eta)N(t)} + u_{\geq c(\eta) N(t)}$.
To estimate the very low frequencies of $u$, we use the full strength of Corollary~\ref{C:finite energy}.  Indeed,
by H\"older's inequality and \eqref{ultimate},
\begin{align*}
\|\nabla u_{\leq M}\|_{L_t^\infty L_x^{\frac4{2+\beta}}} &+ \|P_{\leq M}u_t\|_{L_t^\infty L_x^{\frac4{2+\beta}}}\\
&\lesssim \bigl\|\langle x-x(t)\rangle^\beta \nabla u_{\leq 1}\bigr\|_{L_t^\infty L_x^2}
+ \bigl\|\langle x-x(t)\rangle^\beta P_{\leq 1}u_t\bigr\|_{L_t^\infty L_x^2}\\
&\lesssim_u 1.
\end{align*}
Thus, by Bernstein's inequality,
\begin{align}\label{small freq}
\|\nabla u_{\leq M}\|_{L_t^\infty L_x^2} + \|P_{\leq M}u_t\|_{L_t^\infty L_x^2}
&\lesssim_u M^{\frac{3\beta}4}.
\end{align}

To estimate the medium frequencies in the decomposition of $u$, we use Bernstein's inequality and \eqref{small freq comp}:
\begin{align}\label{med freq}
\|\nabla u_{M\leq \cdot\leq c(\eta) N(t)}\|_{L_t^\infty L_x^2} &+ \|P_{M\leq \cdot\leq c(\eta) N(t)}u_t\|_{L_t^\infty L_x^2}\notag\\
&\lesssim M^{1-s_c} \bigl[\|u_{\leq c(\eta) N(t)}\|_{L_t^\infty \dot H^{s_c}_x} + \|P_{\leq c(\eta) N(t)}u_t\|_{L_t^\infty \dot H^{s_c-1}_x}\bigr]\notag\\
&\lesssim M^{1-s_c}\eta.
\end{align}
We estimate the high frequencies in the decomposition of $u$ similarly:
\begin{align}\label{big freq}
\|\nabla u_{\geq c(\eta) N(t)}\|_{L_t^\infty L_x^2} &+ \|P_{\geq c(\eta) N(t)}u_t\|_{L_t^\infty L_x^2}\notag\\
&\lesssim [c(\eta)N(t)]^{1-s_c} \bigl[\|u_{\geq c(\eta) N(t)}\|_{L_t^\infty \dot H^{s_c}_x} + \|P_{\geq c(\eta) N(t)}u_t\|_{L_t^\infty \dot H^{s_c-1}_x}
\bigr]\notag\\
&\lesssim_u [c(\eta)N(t)]^{1-s_c}.
\end{align}

Putting together \eqref{small freq}, \eqref{med freq}, and \eqref{big freq}, we get
\begin{align}\label{kinetic}
\|\nabla u\|_{L_t^\infty L_x^2} + \|u_t\|_{L_t^\infty L_x^2}
\lesssim_u M^{\frac{3\beta}4} + M^{1-s_c}\eta + [c(\eta)N(t)]^{1-s_c}.
\end{align}
By Sobolev embedding and interpolating between \eqref{kinetic} and the fact that $u\in L_t^\infty \dot H^{s_c}_x$,
we also obtain
\begin{align}\label{potential}
\|u\|_{L_t^\infty L_x^{p+2}}\lesssim_u \bigl[M^{\frac{3\beta}4} + M^{1-s_c}\eta + [c(\eta)N(t)]^{1-s_c}\bigr]^{\frac2{p+2}}.
\end{align}

Combining \eqref{kinetic} and \eqref{potential}, we thus get
$$
E(u)\lesssim_u \bigl[M^{\frac{3\beta}4} + M^{1-s_c}\eta + [c(\eta)N(t)]^{1-s_c}\bigr]^2.
$$

Taking $M$ small, and then $\eta$ small depending on $M$, and then $t$ sufficiently large depending on $\eta$
(and recalling that for a frequency-cascade solution, $\limsup_{t\to \infty}N(t)=\infty$), we may deduce that the energy,
which is conserved, is smaller than any positive constant.  Thus $E(u)=0$ and so $u\equiv 0$.
This contradicts the fact that $u$ is a blowup solution.
\end{proof}

%%%%%%%%%%%%%%%%%%%%%%%%%%%%%%%%%%%%%%%%%%%%%%%%%%%%%%%%%%%%%%%%%%%%%%%%%%%%%%%%%%%%%%%%%%%
%
%
%                                   Section
%
%
%%%%%%%%%%%%%%%%%%%%%%%%%%%%%%%%%%%%%%%%%%%%%%%%%%%%%%%%%%%%%%%%%%%%%%%%%%%%%%%%%%%%%%%%%%%

\section{The soliton-like solution}\label{S:no soliton}

In this section, we preclude the soliton-like solution described in Theorem~\ref{T:enemies}.

\begin{theorem}[Absence of solitons]
There are no soliton-like solutions to \eqref{nlw} in the sense of Theorem~\ref{T:enemies}.
\end{theorem}

\begin{proof}
We argue by contradiction.  Assume there exists a solution $u:\R\times\R^3\to \R$ that is soliton-like in the sense
of Theorem~\ref{T:enemies}.  We will show this scenario is inconsistent with the Morawetz inequality \eqref{Morawetz}.

By Corollary~\ref{C:finite energy}, the soliton has finite energy; hence, the right-hand side
in the Morawetz inequality is finite and so
\begin{align}\label{Mor}
\int_0^T \int_{\R^3} \frac{|u(t,x)|^{p+2}}{|x|}\, dx\, dt\lesssim E(u) \lesssim_u 1,
\end{align}
for any $T>0$.  On the other hand, by Corollary~\ref{C:pot conc} we have concentration of potential energy, that is,
there exists $C=C(u)$ so that
$$
\int_{t_0}^{t_0+1} \int_{|x-x(t)|\leq C} |u(t,x)|^{p+2}\, dx\, dt \gtrsim_u 1,
$$
for any $t_0\in\R$.  Translating space so that $x(0)=0$ and employing finite speed of propagation in the sense of
\eqref{x is Lip}, we obtain that for $T\geq 1$,
\begin{align*}
\text{LHS\eqref{Mor}}
& \geq \int_0^T \int_{|x-x(t)|\leq C} \frac{|u(t,x)|^{p+2}}{|x|}\, dx\, dt
\gtrsim_u \int_0^T \frac{dt}{1+t}
\gtrsim_u \log(T).
\end{align*}
Choosing $T$ sufficiently large depending on $u$, we derive a contradiction to \eqref{Mor}.
\end{proof}

%%%%%%%%%%%%%%%%%%%%%%%%%%%%%%%%%%%%%%%%%%%%%%%%%%%%%%%%%%%%%%%%%%%%%%%%%%%%%%%%%%%%%%%%%%%
%
%
%                                   Section
%
%
%%%%%%%%%%%%%%%%%%%%%%%%%%%%%%%%%%%%%%%%%%%%%%%%%%%%%%%%%%%%%%%%%%%%%%%%%%%%%%%%%%%%%%%%%%%

\section{The finite-time blowup solution}\label{S:ftb}

In this section, we preclude the finite-time blowup scenario described in Theorem~\ref{T:enemies} by showing that such solutions
are inconsistent with the conservation of energy.

\begin{theorem}[Absence of finite-time blowup solutions]
There are no finite-time blowup solutions to \eqref{nlw} in the sense of Theorem~\ref{T:enemies}.
\end{theorem}

\begin{proof}
We argue by contradiction.  Assume there exists a solution $u:I\times\R^3\to \R$ that is a finite-time blowup solution
in the sense of Theorem~\ref{T:enemies}.  By the time-reversal and time-translation symmetries, we may assume that
the solution blows up as $t\searrow 0=\inf I$.

First note that $N(t)\to\infty$ as $t\to0$, for otherwise a subsequential limit of the normalizations $u^{[t]}$ would
blow up instantaneously, in contradiction of the local theory.  Combining this with \eqref{x is Lip}, we deduce that
$\lim_{t\to 0} x(t)$ exists. By space-translation symmetry, we may assume that $\lim_{t\to 0} x(t)=0$.

Next we show that
\begin{align}\label{bounded support 1}
\supp u(t) \cup \supp u_t(t) \subseteq B(0, t) \quad  \text{for all} \quad t\in I,
\end{align}
where $B(0,t)$ denotes the closed ball in $\R^3$ centered at the origin of radius $t$.  Indeed, it suffices to show that
\begin{align}\label{bounded support 2}
\lim_{t\to 0}\int_{t +\eps \leq |x|\leq \eps^{-1} - t} \tfrac12 \bigl| \ntx u(t,x)\bigr|^2 + \tfrac1{p+2} |u(t,x)|^{p+2} \, dx = 0
\quad \text{for all} \quad \eps>0,
\end{align}
because the energy on the annulus $\{x:t +\eps \leq |x|\leq \eps^{-1} - t\}$ is finite and does not decrease as $t\to0$.
To obtain \eqref{bounded support 2}, fix $\eps>0$.  As the parameters $N(t)$ and $x(t)$ satisfy
$$
\lim_{t\to 0} N(t)=\infty \quad \text{and} \quad |x(t)|\leq |t| + C_u N(t)^{-1} \text{ for all } t\in I,
$$
we deduce that for all $\eta>0$ there exists $t_0=t_0(\eps, \eta)$ such that for $0<t<t_0$ we have
$$
\{x\in \R^3:\ t +\eps \leq |x|\leq \eps^{-1} - t\} \subseteq \{x\in \R^3:\ |x-x(t)|\geq C(\eta)/N(t)\},
$$
where $C(\eta)$ is as in \eqref{E:u' compact}.  Thus by H\"older's inequality and \eqref{E:u' compact},
\begin{align*}
&\int_{t +\eps \leq |x|\leq \eps^{-1} - t}   \tfrac12 \bigl| \ntx u(t,x)\bigr|^2 + \tfrac1{p+2} |u(t,x)|^{p+2} \, dx\\
&\quad\lesssim \eps^{\frac4{p}-1} \Bigl[\bigl\|\ntx u(t)\bigr\|_{L_x^{\frac{3p}{p+2}}(\{|x-x(t)|\geq C(\eta)/N(t)\})}^2
    + \|u(t)\|_{L_x^{\frac{3p}2}(\{|x-x(t)|\geq C(\eta)/N(t)\})}^{p+2}\Bigr]\\
&\quad \lesssim \eps^{\frac4{p}-1} \eta^2
\end{align*}
for all $0<t<t_0$.  As $\eta$ can be made arbitrarily small, this proves \eqref{bounded support 2} and hence \eqref{bounded support 1}.

To continue, by \eqref{bounded support 1}, H\"older's inequality, and Sobolev embedding we obtain
\begin{align*}
E(u(t))
&= \int_{B(0,t)} \Bigl(\tfrac12 |\ntx u(t,x)|^2 + \tfrac1{p+2} |u(t,x)|^{p+2}\Bigr) \, dx\\
&\lesssim \Bigl(\|\ntx u(t)\|_{L_x^{\frac{3p}{p+2}}}^2 + \|u(t)\|_{ L_x^{\frac{3p}2}}^{p+2}\Bigr)t^{1-\frac4p}\\
& \lesssim_u t^{1-\frac4p}
\end{align*}
for all $t\in I$.  In particular, the energy of the solution is finite and converges to zero as the time $t$ approaches
the blowup time $0$. Invoking the conservation of energy, we deduce that $u\equiv 0$.  This contradicts the fact that
$u$ is a blowup solution.
\end{proof}

%%%%%%%%%%%%%%%%%%%%%%%%%%%%%%%%%%%%%%%%%%%%%%%%%%%%%%%%%%%%%%%%%%%%%%%%%%%%%%%%%%%%%%%%%%%
%
%
%                                   Section
%
%
%%%%%%%%%%%%%%%%%%%%%%%%%%%%%%%%%%%%%%%%%%%%%%%%%%%%%%%%%%%%%%%%%%%%%%%%%%%%%%%%%%%%%%%%%%%

\end{document}